\documentclass[11pt,a4paper]{article}
\usepackage{geometry}
\usepackage{graphicx}
\usepackage{bm,array}
\usepackage{comment}
\usepackage{caption}
\usepackage{subcaption}
\usepackage{tablefootnote}
 \usepackage{makecell}
\usepackage{booktabs}
 \usepackage{multirow}
 \usepackage{float}

\usepackage[normalem]{ulem}

\usepackage[pdfpagemode={UseOutlines},bookmarks=true,bookmarksopen=true,
bookmarksopenlevel=0,bookmarksnumbered=true,hypertexnames=false,
colorlinks,linkcolor={blue},citecolor={blue},urlcolor={red},
pdfstartview={FitV},unicode,breaklinks=true]{hyperref}
\hypersetup{urlcolor=blue, colorlinks=true}

\usepackage{setspace}
\usepackage{vmargin}
\setmarginsrb{ 1.0in}  % left margin
{ 0.6in}  % top margin
{ 1.0in}  % right margin
{ 0.8in}  % bottom margin
{  20pt}  % head height
{0.25in}  % head sep
{9pt}  % foot height
{ 0.3in}  % foot sep

\usepackage{authblk}

\usepackage{amsmath}
\usepackage{amsfonts}
\usepackage{amssymb}
\usepackage[round, sort,comma,authoryear]{natbib}
\usepackage{mdframed}

\newtheorem{theorem}{Theorem}

\newtheorem{corollary}{Corollary}

\newtheorem{definition}[theorem]{Definition}

\newtheorem{lemma}{Lemma}

\newtheorem{proposition}{Proposition}

\newenvironment{proof}[1][Proof]{\noindent\textbf{#1.} }{\ \rule{0.5em}{0.5em}}

\usepackage{multirow}% http://ctan.org/pkg/multirow
\usepackage{hhline}% http://ctan.org/pkg/hhline
\usepackage{footnote}
\usepackage[ruled,vlined]{algorithm2e}
\usepackage{algorithmic}

\newcommand{\cS}{{\mathcal{S}}}

\newcommand{\cK}{{\mathcal{K}}}
\newcommand{\cF}{{\mathcal{F}}}

\newcommand{\cH}{{\mathcal{H}}}

\newcommand{\cO}{{\mathcal{O}}}

\newcommand{\cZ}{{\mathcal{Z}}}

\newcommand{\wLambda}{{\widetilde{\Lambda}}}
\newcommand{\wTheta}{{\widetilde{\Theta}}}

\newcommand{\bx}{\textbf{x}}
\newcommand{\by}{\textbf{y}}

\newcommand{\ba}{\textbf{a}}

\newcommand{\bz}{\textbf{z}}

\newcommand{\br}{\textbf{r}}

\newcommand{\btheta}{\pmb{\theta}}

\newcommand{\beps}{\pmb{\epsilon}}

\newcommand{\bbR}{\mathbb{R}}
\newcommand{\bbE}{\mathbb{E}}

\newcommand{\OA}{\textsc{OA}}
\newcommand{\IA}{\textsc{IA}}

\usepackage{mathtools}
\DeclarePairedDelimiter\ceil{\lceil}{\rceil}

\usepackage{mathtools}

\newif\ifnotes\notestrue
\def\htien#1{}

\usepackage{xcolor}

\begin{document}
%%%%%%%%%%%%%%%%

%    \RUNTITLE{Joint Facility Location and Cost Optimization}

% \TITLE{Bundling Information Goods of Decreasing Value}
%\TITLE{Joint  Location and Cost Planning in Maximum Capture Facility Location under Multiplicative Random Utility Maximization}

%should be entered in ONE field, separated by a comma.
%\EMAIL field can be repeated if more than one author
%\ARTICLEAUTHORS{%
%\AUTHOR{Ngan Ha Duong, Tien Thanh Dam, Thuy Anh Ta}
%\AFF{Department of Computer Science, Phenikaa University}
%\AUTHOR{Tien Mai}
%\AFF{School of Computing and Information Systems, Singapore Management University, \EMAIL{atmai@smu.edu.sg}}
%} % end of the block

%\ABSTRACT{%
%...
%}%

%\KEYWORDS{deterministic inventory theory; infinite linear programming duality;
%  existence of optimal policies; semi-Markov decision process; cyclic schedule}

%\KEYWORDS{Maximum capture, cost optimization,  multiplicative random utility, mixed-integer program} 
%\HISTORY{This paper wasfirst submitted on April 12, 1922 and has been with the authors for 83 years for 65 revisions.}

%\maketitle
\newcolumntype{C}{>{\centering\arraybackslash}p{4em}}

\title{\textbf{Competitive Facility Location with  Market Expansion and Customer-centric Objective}}
\author[2]{Cuong Le}
\author[1,*]{Tien Mai}
\author[3]{Ngan Ha Duong}
\author[3]{Minh Hoang Ha}
%\author[2]{Ms.B}
\affil[1]{\it\small
School of Computing and Information Systems, Singapore Management University}
\affil[2]{\it\small 
Interdisciplinary Centre for Security, Reliability and Trust, University of Luxembourg}
\affil[3]{\it\small 
SLSCM and CADA, Faculty of Data Science and Artificial Intelligence, College of Technology, National Economics University, Hanoi, Vietnam
}
\affil[*]{\it\small Corresponding author, atmai@smu.edu.sg}

\maketitle

% \pagebreak
% \begin{center}
% 	\linespread{1.5}
% 	\LARGE
% 	\bf
% 	{Joint  Location and Cost Planning in Maximum Capture Facility Location under Multiplicative Random Utility Maximization}
% \end{center}

\begin{abstract}

We study a competitive facility location problem, where customer behavior is modeled and predicted using a discrete choice random utility model. The goal is to strategically place new facilities to maximize the overall captured customer demand in a competitive marketplace. In this work, we introduce two novel considerations. First, the total customer demand in the market is not fixed but is modeled as an increasing function of the customers' total utilities. Second, we incorporate a new term into the objective function, aiming to balance the firm's benefits and customer satisfaction. Our new formulation exhibits a highly nonlinear structure and is not directly solved by existing approaches. To address this, we first demonstrate that, under a concave market expansion function, the objective function is concave and submodular, allowing for a $(1-1/e)$ approximation solution by a simple polynomial-time greedy algorithm. We then develop a new method, called Inner-approximation, which enables us to approximate the mixed-integer nonlinear problem (MINLP), with arbitrary precision, by an MILP without introducing additional integer variables. We further demonstrate that our inner-approximation method consistently yields lower approximations than the outer-approximation methods typically used in the literature. Moreover, we extend our settings by considering a\textit{ general (non-concave)} market-expansion function and show that the Inner-approximation mechanism enables us to  approximate the resulting MINLP, with arbitrary precision, by an MILP. To further enhance this MILP, we show how to significantly reduce the number of additional binary variables by leveraging concave areas of the objective function. Extensive experiments demonstrate the efficiency of our approaches.
\end{abstract}


%%%%%%%%%%%%%%%%%%%%%%%%%%%%%%%%%%%%%%%%%%%%%%%%%%%%%%%%%%%%%%%%%%%%%%

% Samples of sectioning (and labeling) in OPRE
% NOTE: (1) \section and \subsection do NOT end with a period
% (2) \subsubsection and lower need end punctuation
% (3) capitalization is as shown (title style).
%
%\section{Introduction.}\label{intro} %%1.
%\subsection{Duality and the Classical EOQ Problem.}\label{class-EOQ} %% 1.1.
%\subsection{Outline.}\label{outline1} %% 1.2.
%\subsubsection{Cyclic Schedules for the General Deterministic SMDP.}
%  \label{cyclic-schedules} %% 1.2.1
%\section{Problem Description.}\label{problemdescription} %% 2.

% Text of your paper here

\section{Introduction} 

The problem of facility location has been a key area of focus in decision-making for modern transportation and logistics systems for a long time. Typically, it involves selecting a subset of potential sites from a pool of candidates and determining the financial investment required to establish facilities at these chosen locations. The goal is usually to either maximize profits (such as projected customer demand or revenue) or minimize costs (such as operational or transportation expenses). A critical factor in these decisions is customer demand, which significantly influences facility location strategies. In this study, we focus on a specific class of competitive facility location problems, where customer demand is predicted using a random utility maximization (RUM) discrete choice model \citep{McFaTrai00} and the aim is to locate new facilities in a market already occupied by a competitor  \citep{train2009discrete,BenaHans02,mai2020multicut}.  Here, it is assumed that customers choose between available facilities by maximizing the utility they derive from each option. These utilities are typically based on attributes of the facilities, such as service quality, infrastructure, or transportation costs, as well as customer characteristics like age, income, and gender. The use of the RUM framework in this context is well supported, given its widespread success in modeling and predicting human choice behavior in transportation-related applications \citep{Mcfadden2001economicNobel,BenABier99a}.

In the context of competitive facility location under the RUM framework, to the best of our knowledge, existing studies generally assume that the maximum customer demand that can be captured by each facility is fixed and independent of the availability of new facilities entering the market. However, this assumption is limited in many practical scenarios. Intuitively, the total market demand is likely to expand when more facilities are built. Moreover, most of the existing studies focus solely on maximizing the total expected captured demand, ignoring factors that account for customer satisfaction.

An example that highlights the importance of considering such factors is when a new electrical vehicles (EV) company plans to build electric vehicle charging stations to compete with other competitors (such as gas stations or public transport). A critical consideration for the firm is that adding more EV stations in the market could likely expand the EV market, attracting more customers from competitors \citep{Sierzchula2014,Li2017}.
Additionally, in certain cases, building more EV stations in urban areas might help generate more profit by attracting more customers, but this may not be the best long-term strategy. Customers from non-urban areas would have less access to these facilities and may lose interest in adopting EVs, which may hinder broader EV adoption \citep{Gnann2018,Bonges2016}. Thus, for a long-term, sustainable development strategy, the company would need to balance overall profit with customer satisfaction.

Motivated by this observation, in this paper, we explore \textit{two new considerations that better capture realistic customer demand and balance both the company’s profit and customer satisfaction}. Specifically, we assume that the maximum customer demand (i.e., the total number of customers that existing and new facilities can attract) is no longer fixed but modeled as an increasing market expansion function of customer utility. We also introduce a term representing total customer utility to account for customer satisfaction in the main objective function. The resulting optimization problem is highly non-convex, and to the best of our knowledge, no existing algorithm can solve it to optimality, or guarantee near-optimal solutions. To address these challenges, we have developed innovative solution algorithms with theoretical support that can guarantee near-optimality under both concave and non-concave market expansion functions. Our\textbf{ key contributions} are detailed as follows:
\begin{itemize}
    \item \textbf{Problem formulation:} We formulate a competitive facility location problem with market expansion and a customer-centric objective function. The goal is to maximize both the expected captured demand and the total utility of customers (or the expected consumer surplus associated with all the available facilities in the market), assuming that the maximum customer demand for both new and existing facilities is not fixed, but modeled as an increasing function of the customers' total utility value. \textit{The problem is characterized by its high nonlinearity and, to the best of our knowledge, cannot be solved to optimality or near-optimality by existing methods.}
    
    \item \textbf{Concavity and submodularity:} We first examine the problem with concave market expansion functions. We show that, under certain conditions, the objective function is monotonically increasing and submodular. This submodularity property ensures that a simple and fast greedy heuristic can guarantee a $(1 - 1/e)$ approximation solution. It is important to note that submodularity is known to hold in the context of choice-based facility location under a fixed market setting. Our findings extend this result by showing that submodularity also holds under a dynamic market setting with concave market expansion functions.
    
    \item \textbf{Inner-approximation:} For concave market expansion functions, existing exact methods typically rely on outer-approximation techniques that iteratively approximate the concave objective function using sub-gradient cuts. We propose an alternative approach, called inner-approximation, that builds an inner approximation of the objective function using piecewise linear approximations (with arbitrarily small approximation errors). We theoretically show that this inner-approximation approach guarantees smaller approximation errors compared to outer-approximation counterparts. Furthermore, we show that the approximation problem can be reformulated as a mixed-integer linear program (MILP) without additional integer variables, and the number of constraints is proportional to the number of breakpoints used to construct the inner-approximation. We also develop a mechanism to optimize the number of breakpoints (and hence the size of the MILP) for a pre-specified approximation accuracy level.
    
    \item \textbf{General non-concave market expansion:} We take a significant step toward modeling realistic market dynamics by considering the facility location problem with a general non-concave market expansion function. We adapt the ``inner-approximation'' approach to approximate the resulting mixed-integer non-concave problem into a MILP with additional binary variables. By identifying intervals where the objective function is either concave or convex, we relax part of the additional binary variables, enhancing the performance of the MILP approximation. We also optimize the selection of breakpoints for constructing the piecewise linear approximations under this general market expansion setting.
    
    \item \textbf{Experimental validation:} We provide extensive experiments using well-known benchmark instances of various sizes to demonstrate the efficiency of our approaches, under both concave and non-concave market expansion functions.
\end{itemize}

\noindent \paragraph{Paper Outline:} The paper is structured as follows: Section \ref{sec:formulation} introduces the problem formulation. Section \ref{sec:submodular} discusses the submodularity of the objective function in the context of concave market expansion functions. In Section \ref{sec:outer-inner}, we present our inner-approximation solution method. Section \ref{sec:general non-concave} addresses our approaches for the facility location problem with a general non-concave market expansion function. Section \ref{sec:experiments} presents the numerical results, while Section \ref{sec:concl} concludes the paper. Additional proofs and further details not covered in the main body are provided in the appendix.

\noindent
\textbf{Notation:}
Boldface characters represent matrices (or vectors), and $a_i$ denotes the $i$-th element of vector $\ba$. We use $[m]$, for any $m\in \mathbb{N}$, to denote the set $\{1,\ldots,m\}$.

\section{Literature Review}
Competitive facility location under random utility maximization (RUM) models has been a topic of interest in Operations Research and Operations Management for several decades. This area of research differentiates itself from other facility location problems through the use of discrete choice models to predict customer demand, drawing from a well-established body of work on discrete choice modeling \citep{train2009discrete}. In the context of competitive facility location (CFL) under RUM models, most studies adopt the Multinomial Logit (MNL) model to represent customer demand. Notably, \cite{BenaHans02} were among the first to introduce the CFL problem under the MNL model, utilizing a Mixed-Integer Linear Programming (MILP) approach that combines a branch-and-bound procedure for small instances with a simple variable neighborhood search for larger instances. 

Subsequent contributions include alternative MILP models proposed by \cite{Zhang2012} and \cite{Haase2009}. \cite{haase2014comparison} conducted a benchmarking study of these MILP models, concluding that \cite{Haase2009}'s formulation exhibited the best performance. \cite{Freire2015} enhanced \cite{Haase2009}'s MILP model by incorporating tighter inequalities into a branch-and-bound algorithm. Additionally, \cite{Ljubic2018outer} developed a Branch-and-Cut method that combines outer-approximation and submodular cuts, while \cite{mai2020multicut} introduced a multicut outer-approximation algorithm designed for efficiently solving large instances. This method generates outer-approximation cuts for groups of demand points rather than for individual points.

A few studies have also explored CFL using more general choice models, such as the Mixed Multinomial Logit (MMNL) model \citep{Haase2009,haase2014comparison}. However, applying the MMNL model typically requires large sample sizes to approximate the objective function, leading to complex problem instances. \cite{dam2022submodularity,dam2023robust} incorporated the Generalized Extreme Value (GEV) family into CFL and proposed a heuristic method that outperforms existing exact methods. \cite{mendez2023follower} investigated CFL under the Nested Logit (NL) model, proposing exact methods based on outer-approximation and submodular cuts within a Branch-and-Cut procedure. Recently, \cite{le2024competitive} explored CFL under the Cross-Nested Logit model, considered one of the most flexible discrete choice models in the literature. In their work, the authors demonstrated that, although the objective function is not concave, it can be reformulated as a mixed-integer concave program, allowing the use of standard exact methods like outer-approximation.

In all the aforementioned studies, the market size is assumed to be fixed and independent of the customer’s total utility. Additionally, these works focus solely on maximizing expected captured demand, neglecting factors related to customer satisfaction. On the other hand, because the objective function in most cases can either be shown to be concave or reformulated as a concave program, outer-approximation methods \citep{mai2020multicut,duran1986outer} have remained the state-of-the-art approaches.  Our work, therefore, makes a significant advancement in this literature by introducing a novel problem formulation that accounts for both market dynamics and customer satisfaction. Furthermore, we propose a new near-exact approach based on inner-approximation, which guarantees smaller approximation errors compared to traditional outer-approximation methods.

Our work and the general context of choice-based competitive facility location are related to a body of research on competitive facility location where customer behavior is modeled using gravity models \citep{DREZNER2002, ABOOLIAN2007a, ABOOLIAN2007b, ABOOLIAN2021, lin2022locating}. These models, in their classical form without market expansion and customer objective components, share a similar objective structure with the CFL problem under the MNL model. Market expansion perspectives have also been considered in this line of work \citep{ABOOLIAN2007a, ABOOLIAN2007b, lin2022locating}. However, since these studies rely on different customer behavior assumptions, the form of the customer's total utility significantly differs from the total utility function under the discrete choice models considered in our work.  Moreover, while these works are restricted to concave market expansion functions, our work considers both concave and non-concave functions, allowing broader applications. In terms of methodological developments, while prior work employs outer-approximation approaches to handle the nonlinear concave demand function, we explore a new type of approximation based on ``\textit{inner-approximation}''. This approach not only offers smaller approximation gaps but also allows efficient solving of problems with general non-concave market expansion functions.

\section{Problem Formulation}\label{sec:formulation}
In the classic facility location, decision-makers  aim to establish new facilities in a manner that optimizes the demand fulfilled from customers. However, accurately assessing customer demand in real-world scenarios is challenging and inherently uncertain. In this study, we  study  a facility location problem where discrete choice models are used to estimate and predict customer demand. Among various approaches discussed in demand modeling literature, the Random Utility Maximization (RUM) framework \citep{train2009discrete} stands out as the most prevalent method for modeling discrete choice behaviors. This method is grounded in the random utility theory, positing that a decision-maker's preference for an option is represented through a random utility. Consequently, the customer tends to opt for the alternative offering the highest utility. According to the RUM framework \citep{McFa78,FosgBier09}, the likelihood of individual $n$ choosing option $i\in S$ is determined by $P(u_{ni}\geq u_{nj},;\forall j\in S)$, implying that the individual will select the option providing the highest utility. Here, the random utilities are typical defined as $u_{ni}=v_{ni} + \epsilon_{ni}$, where $v_{ni}$ represents the deterministic component, which can be calculated based on the characteristics of the alternative and/or the decision-maker and some parameters to be estimated, and $\epsilon_{ni}$ represent random components that are  unknown to the analyst. Under the popular Multinomial Logit (MNL) model, the probability that a facility located at position $i$ is chosen by an individual $n$ is computed as  $P_n(i|S) = \frac{e^{v_{ni}}}{\sum_{i\in S}e^{v_{ni}}}$, where $S$ is the set of available facilities.

In this study, we consider a competitive facility location problem where a ``\textit{newcomer}'' company plans to enter a market already  captured by a competitor (e.g., an electrical vehicle (EV) company is aiming to break into the transportation market, which is currently dominated by companies offering gasoline-powered vehicles or other EV brands.). The main objective is to secure a portion of the market share by attracting customers to their newly opened facilities. To forecast the impact of these new facilities on customer demand, we  employ the RUM framework, which assumes that each customer assigns a random utility to each facility (both the newcomer's and competitors') and makes decisions aimed at maximizing their personal utility. Consequently, the company's strategy revolves around selecting an optimal set of locations for its new facilities to maximize the anticipated customer footfall. 

To describe the mathematical formulation of the problem, let $[m]$ be the set of available locations, $[N]$  be the set of customer types available in the market, whereas a customer's type can be defined by  geographic locations. Moreover, let $v_{ni}$ be the utility of facility located at location $i\in [m]$ associated with customer type $n\in [N]$, and $\cS^c$  be the set of competitor's facilities. We also denote $q_n$ be the \textit{{maximum customer expenditure}} in zone $n\in [N]$. 
Given a location decision  $S\subseteq [m]$, i.e., set of chosen locations and under the MNL choice model, the choice probability of  a new facility $i\in [m]$ is  given as: 
\[
P(i\Big|S \cup \cS^c ) = \frac{e^{v_{ni}}}{\sum_{j\in S} e^{v_{nj}} +\sum_{j\in \cS^c} e^{v_{nj}}}.
\]
The competitive facility location problem, in its classical form, can be formulated as:
\begin{align}
    \max_{S} & \qquad\sum_{n\in [N]} q_n\sum_{j\in S} P(i\Big|S \cup \cS^c ) \nonumber\\
    \text{s.t.} & \qquad |S| \leq C.
\end{align}
The above formulation has been widely employed in the context  of choice-based facility location \citep{BenaHans02,FLO_Hasse2009MIP,Ljubic2018outer,mai2020multicut}.
 This formulation, however, presumes that the total demand for customer type $n$ (that is, $q_n$) remains constant, regardless of an increase in demand as more facilities become available in the market. Additionally, this formulation does not consider customer satisfaction, which is likely to enhance with the availability of more facilities in the market. To address these shortcomings, let us consider the following  customer's expected utility function as a function of the chosen locations $S$,  under the assumption that customers make choices according to the MNL model \citep{train2009discrete}:  
\begin{align*}
    \Psi_n(S) &= \bbE_{\beps} \left[\max_{i\in S\cup \cS^c}\Big\{ v_{ni}+\epsilon_{ni}\Big\} \right] = \log\left(\sum_{i\in S}  e^{v_{ni}} + \sum_{j\in \cS^c} e^{v_{nj}}\right).
\end{align*}
The function \(\phi_n(S)\) represents the \textit{expected utility} experienced by customers of type \(n\) when the available facilities in the market are those in the set \(S \cup \mathcal{S}^c\). This function is commonly referred to as the \textit{expected consumer surplus} associated with the choice set \(S \cup \mathcal{S}^c\). It captures the \textit{inclusive value} of the choice set of available facilities, reflecting the combined attractiveness of all available alternatives within it \citep{train2009discrete,daly1978logsum}.

It is to be expected that the total demand of customers would be a increasing function of $\phi_n(n)$, since an increase in customer utilities should be likely to attract more customers to the market.  With this consideration, we introduce the following formulation that enable us to capture both market expansion and customer-centric  values in the objective function.
\begin{align}
     \max_{S \subseteq [m]} &\left\{\cF(S) =\sum_{n\in [N]} q_n g(\phi_n(S)) \left(\sum_{i\in S} P(i\Big|~ S\cup \cS^c)\right) + \sum_{n}\alpha_n \phi_n (S)\right\} \label{eq:prob-1} \\
    \mbox{subject to} &\quad |S|\leq C \nonumber
\end{align}
where $g(t)$ is an increasing function that reflects the impact of customers' expected utilities on market expansion (namely, customers' expenditures), and $\alpha_n$ represent specific parameters. These parameters, $\alpha_n$, are scalar values that help quantify the balance between the firm's captured demand and the expected utility for customers. An increase in $\alpha_n$ would enhance customer satisfaction, but might negatively influence the firm's captured demand, and vice versa.  Furthermore, a location solution $S$ that boosts the customer's expected utility $\phi_n(S)$ will also attract more customers, thereby expanding the overall market via the increasing function $g_n(\cdot)$. For notational simplicity, we include only a basic cardinality constraint on the number of open facilities, \( |S| \leq C \), while noting that our approach is general and capable of handling any linear constraints.

It is convenient to formulate \eqref{eq:prob-1} as a binary program. To simplify notation, let's first denote $V_{ni} = e^{v_{ni}}$ and $U^c_n = \sum_{j\in \cS^c} V_{nj}$. We then reformulate \eqref{eq:prob-1} as the following nonlinear program:
\begin{align}
     \max_{\bx} &\quad \Bigg\{\sum_{n\in [N]} q_n g\left(\log\left(\sum_{i\in [m]} x_iV_{ni} + U^c_n\right)\right)   \left( \frac{ \sum_{i\in [m]} x_i V_{ni}}{U^c_n + \sum_{i\in [m]} x_i V_{ni}}\right) \nonumber\\
     &~~~~~\qquad\qquad\qquad + \sum_{n}\alpha_n \log\left(\sum_{i\in [m]} x_iV_{ni} + U^c_n\right) \Bigg\} \label{prob:ME-MCP-1}\\
    \mbox{subject to} &\quad \sum_{i \in [m]} x_i \leq C \nonumber\\
    &\quad \bx \in \{0,1\}^{m} \nonumber.
\end{align}
We  refer to the problem as the maximum capture problem with market expansion (ME-MCP).   By further letting  $z_n = U^c_n + \sum_{i\in [m]} x_iV_{ni}$
\footnote{Previous works typically assume that $U^c_n = 1$ for all $n \in [N]$ for ease of notation, without loss of generality \citep{mai2020multicut,dam2022submodularity}. This is possible because we can divide both the numerator and denominator of each fraction in \eqref{prob:ME-MCP-1} to normalize $U^c_n$ to one. However, this approach is not applicable in our context as it would affect the total expected utility $U^n_c + \sum_{i \in S} V_{ni}$.}. 
 We now write rewrite  \eqref{prob:ME-MCP-1} in a more compact form as follows:
\begin{align}
     \max_{\bx,\bz} &\Bigg\{F(\bz) =\sum_{n\in [N]} q_n g\left(\log\left(z_n\right)\right)   \left( \frac{z_n - U^c_n}{z_n}\right)  + \sum_{n}\alpha_n \log\left(z_n\right) \Bigg\} \label{prob:ME-MCP-main} \tag{\sf ME-MCP}\\
    \mbox{subject to} &\quad \sum_{i \in [m]} x_i \leq C \nonumber\\
    & \quad z_n = U^c_n + \sum_{i\in [m]} x_iV_{ni} \nonumber \\
    &\quad \bx \in \{0,1\}^{m},~~\bz \in \bbR^n \nonumber.
\end{align}
In the context of choice-based competitive facility location, without the market expansion term $g(\log(z_n))$ and the customer-centric term $\alpha_n \log(z_n)$, existing solutions  typically rely on the objective function being concave in $\mathbf{x}$ and submodular, enabling exact  solutions via outer-approximation methods,  or rapid identification of good solutions with approximation guarantees through the use of greedy location search algorithms \citep{Ljubic2018outer,mai2020multicut,Dam2021submodularity}. This approach prompts the question of whether such convexity and submodularity properties  remain preserved in our new model with the market expansion and customer-centric terms. We will investigate this matter further in the next section.

To effectively and reasonably address market expansion, it is reasonable to assume  that the market-expansion function $g(t)$ exhibits an increasing behavior  in $t$, as an increase in customers' utilities typically fosters market growth. Additionally, it is essential that $\lim_{t\rightarrow \infty} g(t) = 1$, ensuring that total demand does not surpass the maximum customer expenditure, i.e. $q_n$. Commonly utilized function forms in the literature of market expansion include $g(t) = \frac{t}{t+\alpha}$ and $g(t) = 1 - \alpha e^{-\beta t}$ \citep{ABOOLIAN2007a,lin2022locating}, both of which exhibit concavity in $t$. Thus, in the subsequent section, our primary focus will be on solving the facility location problem under  concave market expansion functions $g(t)$, followed by an exploration for addressing the problem under more general, non-concave market expansion functions. 
%%%%%%%%%%%%%%%%%%%%%%%%%%%%%%%
%%%%%%%%%%%%%%%%%%%%%%%%%%%%%%%

\section{Concavity and Submodularity}\label{sec:submodular}
In this section, we focus on the setting that the \textbf{market expansion function $g(t)$ is concave}, delving into the question of under which conditions the overall objective function is concave and submodular, enabling the use of some efficient outer-approximation and local search algorithms. Specifically, we will first establish conditions for the market expansion function 
$g(\cdot)$ under which the objective function $F(\bz)$ is concave in $\bz$. We will further show that under these conditions, the objective function  $
\cF(S)$ (the objective function defined in terms of a subset selection) is monotonically increasing and submodular. As a result, 
\eqref{prob:ME-MCP-main} can be conveniently solved by outer-approximation or local search methods. We further leverage the fact that 
$F(z)$ is univariate to explore an inner-approximation mechanism which allows us to approximate 
\eqref{prob:ME-MCP-main} by a MILP with arbitrary precision. We will theoretically prove that this inner-approximation approach always yields small approximation errors, as compared to an  outer-approximation approach.

From the formulation in \eqref{prob:ME-MCP-main}, we first consider function $\Psi_z(z_n)$,  for any $n\in [N]$, defined as follows:
\[
\Psi_n(z_n) =  q_n g(\log(z_n))   \left( \frac{ z_n - U^c_n}{z_n}\right) + \alpha_n \log(z_n).
\]
This is an univariate function of $z_n$, depending on the market expansion function $g(t)$. 
In the following theorem, we first state conditions under which $\Psi_n(z_n)$ the objective function $\cF(\bz)$ are concave in $\bz$.
\begin{theorem}\label{th:concavity}
Assume that  $g(t)$ is non-decreasing and concave in $t\in \bbR_+$,  and $g(0)-g'(0)\leq 0$, then $\Psi_n(z_n)$ is concave in $\bz$ and, consequently, $F(\bz)$ is concave in $\bz$.
\end{theorem}

Given two popular forms $g(t) = \frac{t}{t+\alpha}$ and $g(t) = 1 - \beta e^{-\alpha t}$, for $\alpha, \beta > 0$, Proposition \ref{prop:cond-2-functions} establishes conditions for $\alpha$ and $\beta$ that ensure $\Psi_n(z_n)$ exhibits concavity with respect to $z_n$.

\begin{proposition}\label{prop:cond-2-functions}
 $\Psi_n(z_n)$ is concave in $z_n$ if $g(t)$ is chosen as follows:
\begin{itemize}
    \item $g(t) = \frac{t}{t+\alpha}$, for any $\alpha\geq 0$, or 
    \item $g(t) = 1 - \beta \exp(-\alpha t)$, when $\alpha, \beta > 0$ and $(\alpha + 1) \beta > 1$
\end{itemize}
\end{proposition}

The proposition can be verified straightforwardly. The concavity of $\Psi_n(z_n)$  implies that  the objective function in \eqref{prob:ME-MCP-main} is also concave, enabling exact methods such as an outer-approximation algorithm \citep{duran1986outer,mai2020multicut} to be applied. Second, leveraging the concavity, we can further demonstrate that  the objective function of \eqref{prob:ME-MCP-main}, when defined as a subset function,  is submodular. To prove this result, let us consider the objective function defined as a set function in \eqref{eq:prob-1}, which can be written as:
\[
\cF(S) = \sum_{n\in [N]} q_n g\left(\log\left(U^c_n+\sum_{i\in S} V_{ni}\right)\right)   \left( \frac{\sum_{i\in S} V_{ni}}{U^c_n+ \sum_{i\in S} V_{ni}}\right) + \sum_{n}\alpha_n \log \left(U^c_n+\sum_{i\in S} V_{ni}\right).
\]
The following theorem demonstrates that the conditions used in Theorem \ref{th:concavity}, which ensure that \(\cF(\bx)\) is concave with respect to \(\bx\), are also sufficient to guarantee that \(\cF(S)\) is submodular.

\begin{theorem}\label{th:submodular}
If the assumption in Theorem \ref{th:concavity} holds, then  $\cF(S)$ is monotonic increasing  and sub-modular.
\end{theorem}

The proof, which explicitly leverages the concavity of $\cF(\bx)$ to verify submodularity, is provided in the appendix. A direct consequence of the submodularity shown in Theorem \ref{th:submodular} is that a simple polynomial-time greedy algorithm can always return $(1-1/e) \approx 0.6321$ approximation solutions. Such a greedy algorithm can be executed by starting from the null set and adding locations one at a time, choosing at each step the location that increases $\cF(S)$ the most. This phase finishes when we reach the maximum capacity, i.e., $|S| = C$. This greedy procedure  can run in $(mC \tau)$ time, where $\tau$ is the computing time to evaluate $\cF(S)$ for a given subset $S\subseteq [m]$.  Due to the monotonicity and submodularity, if $\overline{S}$ is a solution returned by the above greedy procedure, then it is guaranteed that $\cF(\overline{S}) \geq (1-1/e) \max_{S,~ |S|\leq C}\cF(S)$ \citep{Nemhauser1978}. We state this result in the following corollary.

\begin{corollary}
If the assumption in Theorem \ref{th:concavity} holds, then a greedy heuristic can guarantee a $(1-1/e)$ approximation solution to \eqref{prob:ME-MCP-main}.  
\end{corollary}

\section{Outer and Inner Approximations}\label{sec:outer-inner}
\label{sec:solutions_method}
In this section, we discuss two methods—exact or near-exact—for solving \eqref{prob:ME-MCP-main}, taking advantage of the concavity property outlined in Theorem \ref{th:concavity}. Specifically, we will briefly introduce the outer-approximation method, widely recognized in the literature for addressing mixed-integer nonlinear programs with convex objectives and constraints \citep{duran1986outer,mai2020multicut}. Additionally, we explore an approximation approach that allows solving \eqref{prob:ME-MCP-main} to near-optimality (with arbitrary precision) by approximating it by a MILP. 

\subsection{Outer-approximation}
\label{subsec:outer_approx}
The outer-approximation method \citep{duran1986outer,mai2020multicut,fletcher1994solving} is a well-known approach for solving nonlinear mixed-integer linear programs with convex objective functions and convex constraints. A multi-cut outer-approximation algorithm can execute by building piece-wise linear functions that outer-approximate each nonlinear component of the objective functions (or constraints). In the context of the ME-MCP, this can be done by rewriting \eqref{prob:ME-MCP-main} equivalently as:
\begin{align}
     \max_{\bx} \quad& \sum_{n\in[N]}\theta_n \\
    \mbox{subject to} \quad& \theta_n \leq \Psi_n(z_n), \forall n\in[N]\label{ctr:nonlinear-1}\\
    & z_n = 1 + \sum_{j\in [m]} x_iV_{ni}\nonumber \\
    &\sum_{i\in[m]}x_i \leq C\nonumber\\
    &\bx \in \{0,1\}^m.\nonumber
\end{align}
 Since $\Psi_n(z_n)$ is concave in $z_n$, it is well-known that, for any $\overline{z}_n\geq 0$, $\Phi(z_n) \leq \Psi_n(\overline{z}_n) + \Psi'_n(\overline{z}_n)(z_n - \overline{z}_n)$, for any $z_n >0$. This implies that, for any $\overline{z}_n >0$, the following inequality is valid for the ME-MCP: $\theta_n \leq \Psi_n(\overline{z}_n) + \Psi'_n(\overline{z}_n)(z_n - \overline{z}_n)$. Such valid inequalities are typically refereed to as outer-approximation cuts.  It  then follows that one can replace the nonlinear constraints \eqref{ctr:nonlinear-1} by \textit{sub-gradient} cuts  $\theta_n \leq \Psi_n(\overline{z}_n) + \Psi'_n(\overline{z}_n)(z_n - \overline{z}_n)$ for all $\overline{z}_n$ in the feasible set. The multi-cut outer-approximation is an iterative cutting-plane procedure where, at each iteration, a master problem is solved, with \eqref{ctr:nonlinear-1} being replaced by linear cuts. After each iteration, let $(\overline{\btheta},\overline{\bz},\overline{\bx})$ be a solution candidate obtained from solving the master problem. The algorithm then checks if the nonlinear constraints \eqref{ctr:nonlinear-1} are feasible within an acceptance threshold (denoted as $\epsilon$), i.e., if $\overline{\theta}_n \leq \Psi_n(\overline{z}_n) +\epsilon$, for a given threshold $\epsilon>0$. If this condition holds true, the algorithm terminates and returns $(\overline{\btheta},\overline{\bz},\overline{\bx})$. Otherwise, outer-approximation cuts based on $(\overline{\btheta},\overline{\bz},\overline{\bx})$ are generated and added to the master problem to proceed to the next iteration. It can  be shown that the above procedure is guaranteed  to terminate after a finite number of iterations and return an optimal solution to \eqref{prob:ME-MCP-main} \citep{duran1986outer}.  

In the approach described above, outer-approximation is performed as an iterative \textit{cutting-plane} process, where outer-approximation (or sub-gradient) cuts are iteratively added to a master problem, which takes the form of a MILP. In the context of the competitive facility location problem, an outer-approximation approach can also be implemented differently using a \textit{piecewise linear approximation} method \citep{ABOOLIAN2007a, ABOOLIAN2007b}. In this approach, the univariate and concave functions $\Phi_n(\overline{z}_n)$ are approximated by piecewise linear concave functions. Specifically, the approximation of $\Phi_n(\overline{z}_n)$ is achieved by constructing sub-gradient cuts based on a set of breakpoints within the range of $\overline{z}_n$. This method is also referred to as the \textit{Tangent Line Approximation (TLA)}. Compared to the cutting-plane method mentioned earlier, this approach offers several advantages. Notably, it allows the nonlinear program in \eqref{ctr:nonlinear-1} to be reformulated as a single MILP (with arbitrary precision) without introducing additional binary variables. This MILP can then be solved in one step to obtain a near-optimal solution, whereas the cutting-plane method requires solving a sequence of MILPs.

While outer-approximation in the form of cutting-plane methods has demonstrated state-of-the-art results in the context of competitive facility location without market expansion \citep{mai2020multicut, Ljubic2018outer}, it is no longer an exact method when market expansion considerations are introduced, particularly when the market expansion function is non-concave. Therefore, in the following, we investigate outer (and inner) approximation approaches in the form of piecewise linear approximations. As mentioned earlier, this approach offers the advantage of approximating \eqref{prob:ME-MCP-main} as a single MILP. This not only simplifies the problem structure but also provides a practical and efficient way to handle non-concave market expansion functions.

% This method exploit the fact that the continuous relaxation of the problem \ref{prob:main} is concave and differentiable. 
% We consider the following equivalent problem:
% \begin{align*}
%      \max_{\bx} \quad& \sum_{n\in[N]}\delta_n \\
%     \mbox{subject to} \quad& \delta_n \leq \Psi_n(\bx^*) + \nabla \Psi_n(\bx^*)^T(\bx - \bx^*), \forall{n\in[N],\bx^*\in \{0, 1\}^m}\\
%     &\sum_{i\in[m]}x_i = H\\
%     &\bx \in \{0,1\}^m
% \end{align*}

%\mtien{ ---- Discuss B\&C and submodular cuts. ----}

\subsection{Inner versus Outer Approximations}
In the aforementioned outer-approximation (OA) approach, the mixed-integer nonlinear problem is tackled by approximating each concave component $\Psi_n(z_n)$ with concave piece-wise linear functions in $z_n$, enabling the solution of the ME-MCP through a sequence of MILPs. While achieving state-of-the-art performance in the context of the MCP, this outer-approximation approach is incapable of handling non-concave objective functions, becoming heuristic when the objective function is no-longer concave. In this section, we explore an alternative approach, called piece-wise linear inner-approximation (PWIA), which facilitates solving the ME-MCP by constructing piece-wise linear functions that inner-approximate $\Psi_n(z_n)$ internally. Our PWIA approach offers two advantages. First, as demonstrated later, such an inner-approximation function always yields smaller approximation errors compared to its outer-approximation counterpart. Second, as elucidated in the following section, under a \textit{general non-concave market expansion function}, PWIA allows us to approximate the ME-MCP (with arbitrary precision) via MILPs, rendering it convenient for near-optimal solutions.

To facilitate our later exposition, let us first introduce  formal definitions of piece-wise linear inner and outer approximations as below:
\begin{definition}
For a concave function $\Phi(t):[L,U]\rightarrow\bbR$, the piece-wise linear function created by $K$ linear functions $\{a_k t+ b_k, k\in [K]\}$, defined as $\Gamma(t) = \min_{k\in[K]}\{a_k t + b_k\}$, is termed an outer approximation of $\Phi(t)$ in $[l,U]$ if (and only if) $\Gamma(t) \geq \Phi(t)$ for all $t\in [L,U]$. Conversely, it is considered an inner approximation of $\Phi(t)$ in [L,U] if (and only if) $\Gamma(t)\leq \Phi(t)$ for all $t\in [L,U]$.
\end{definition}

Now, given a concave piece-wise linear approximation function $\Gamma(t) = \min_{k\in[K]}\{a_k t + b_k\}$, let $\{(t_1,\Gamma(t_1));\ldots;(t_H,\Gamma(t_H))\}$  be $H$ ``breakpoints''   of $\Gamma(t)$, i.e. points  where the function transitions from one linear segment to another within its piece-wise structure, such that  $L = t_1<t_2<\ldots<t_H=U$ . Such  breakpoints can be founds by considering all the intersection points of all pairs of the linear functions  $\{a_k t + b_k,~k\in [K]\}$  and select points $(t^*,\Gamma(t^*))$ such that $\Gamma(t^*) \leq a_kt^*+b_k$ for all $k\in [K]$.   The piece-wise linear function $\Gamma(t)$ can be  equivalently represented  as:
\[
\Gamma(t) = \min_{h\in [H-1]} \left\{\Gamma(t_h) + \frac{\Gamma(t_{h+1}) - \Gamma(t_{h})}{t_{h+1} - t_h} (t-t_h) \right\}. 
\]
It can be seen that $H$ is the minimum linear segments necessary to represent $\Gamma(t)$ in $[L,U]$.  We are now ready to state our result saying  that, given any piece-wise linear function $\Gamma(t)$ that outer-approximate a concave function, there are always another piece-wise linear approximation function that inner-approximates that concave function with the same number of necessary line segments, but yields smaller approximation errors.  
approximation gaps. We state this result in the following theorem.
\begin{theorem}\label{th:inner-outer-approx-gap}
Given any concave function $\Phi(t):[L,U]\rightarrow \bbR$ , let $\Gamma^{\OA}(t)$ be a piece-wise linear outer-approximation of $\Phi(t)$ in $[L,U]$, then there always exists a piece-wise linear inner-approximation $\Gamma^{\IA}(t)$ of $\Phi(t)$ with the same number of necessary line segments such that:
\begin{equation}\label{eq:thr-inner-outer}
\max_{t \in [L,U]}|\Phi(t) - \Gamma^{\IA}(t)| \leq \max_{t \in [a,b]}|\Phi(t) - \Gamma^{\OA}(t)|.     
\end{equation}

\end{theorem}
The proof can be found in the appendix, which highlights that the inequality in \eqref{eq:thr-inner-outer} is active (i.e., equality holds) only when the concave function $\Phi(t)$ exhibits uniform curvature across the interval \([L, U]\). This condition occurs if $\Phi(t)$ is either a linear function or takes the shape of a circle. The theorem and its proof further imply that for any piecewise linear outer-approximation of $\Phi(t)$, it is always possible to construct breakpoints within \([L, U]\) that yield a piecewise linear inner-approximation with a smaller approximation gap and the same number of line segments.

Later, we will demonstrate that such a piecewise linear approximation enables reformulation of the original problem as a MILP, with its size generally proportional to the numbers of line segments. Thus, the use of an inner-approximation proves to be more advantageous compared to its outer-approximation counterpart, particularly in terms of computational efficiency.
% Figure \ref{fig:IO-OA} below illustrates that IA is better than OA yielding smaller 

% \begin{figure}[htb]
%     \centering
%     \includegraphics[width=0.7\textwidth]{IO.pdf}
%     \caption{Piece-wise Inner-Approximation provides small gaps than Outer-Approximation}
%     \label{fig:IO-OA}
% \end{figure}

\subsection{MILP Approximation via Inner-Approximation}
We begin by presenting a MILP approximation of \eqref{prob:ME-MCP-main}, where the nonlinear components are approximated using inner-approximation techniques. Following this, we discuss an approach to optimally select the linear segments for the inner-approximation, aiming to minimize the size of the resulting MILP formulation while ensuring a certain level  of approximability.

\subsubsection{MILP Approximation.}
Now we show how to approximate \eqref{prob:ME-MCP-main} as a MILP using inner-approximation functions. We first let $L_n$ and $U_n$ be an upper bound and lower bound of $z_n$ in its feasible set. Such bounds can be estimated quickly by sorting $V_{ni}$, $i\in [m]$, in ascending order and select $C$ first elements for the lower bound, and  $C$ last elements  for the upper bound. This is possible because, if  $\sigma_1,\ldots,\sigma_m$  is a permutation of  $(1,\ldots,m)$ such that $V_{n\sigma_1}\leq \ldots \leq V_{n\sigma_m}$, then the following always holds true:
\[
 1+\sum_{i =1}^{C}( V_{n\sigma_i}) \leq 1+ \sum_{i\in [m]} x_iV_{ni} \leq 1 + \sum_{i =m-C+1}^{m}( V_{n\sigma_i}) 
\]
for all $\bx \in\{0,1\}^m$ such that $\sum_{i}x_i = C$. We can then select the lower and upper bounds for $z_n$ as $L_n =  1+\sum_{i =1}^{C}( V_{n\sigma_i}) $ and $U_n = 1+ \sum_{i =m-C+1}^{m}( V_{n\sigma_i})$. 

To construct piece-wise linear functions that inner-approximate  each component $\Psi_n(z_n)$ of the objective function, we split $[L_n; U_n]$ into $K_n$ sub-intervals $[c^n_k;c^n_{k+1}]$ for $k \in [K_n]$, where $c^n_k,~ k\in [K_n+1]$, are breakpoints such that  $L_n = c^n_1 < c^n_2<\ldots< c^n_{K_n+1} = U_n$. We define the following  piece-wise concave linear function:
\[
\Gamma_n(z)  =  \min_{k\in [K_n-1]} \left\{\Psi_n(c^n_k) + \frac{\Psi_n(c^n_{k+1}) - \Psi_n(c^n_{k})}{c^n_{k+1} - c^n_k} (z-c^n_k) \right\},~ \forall n\in[N].    
\]
We then approximate each concave function $\Psi_n(z_n)$ by $\Gamma_n(z_n) 
$, resulting in the following mixed-integer nonlinear problem:
\begin{align}
     \max_{\bx} \quad& \left\{\sum_{n\in[N]}\theta_n \right\} \label{prob:main-2}\tag{\sf APPROX-1}\\
    \mbox{subject to} \quad& \theta_n \leq \Gamma_n(z_n), \forall n\in[N]\nonumber\\
    & z_n = U^c_n + \sum_{j\in [m]} x_iV_{ni}\nonumber \\
    &\sum_{i\in[m]}x_i \leq C\nonumber\\
    &\bx \in \{0,1\}^m\nonumber.
\end{align}
We then can see that \eqref{prob:main-2} can be reformulated as a MILP with no additional binary variables (Proposition \ref{prop:concave-milp} below).
\begin{proposition}\label{prop:concave-milp}
The MINLP \eqref{prob:main-2} is equivalent to the following MILP:
\begin{align}
     \max_{\bx,\bz,\btheta} &\left\{\sum_{n\in [N]}\theta_n\right\} \label{prob:milp-IA}\tag{\sf IA-MILP} \\
    \mbox{subject to} 
    &\quad \theta_n \leq  \Psi_n(c^n_k) +\frac{\Phi_n(c^n_{k+1}) - \Psi_n(c^n_{k}) }{c^n_{k+1}-c^n_{k}} (z_n - c^n_k),~\forall k\in [K_n-1],~ n\in [N]\nonumber \\
    &\quad z_n = \sum_{i\in [m]} x_i V_{ni} + 1,~ \forall n\in [N]\nonumber \\
    &\quad \sum_{i \in [m]} x_i \leq C \nonumber\\
    &\quad \bx \in \{0,1\}^{m} \nonumber.
\end{align}
\end{proposition}

% Let define the linear cut associated with interval $[c^n_{k};c^n_{k+1}]$
% \[
% \Gamma^k_n(z) =  \Psi_n(c^n_k) +\frac{\Phi_n(c^n_{k+1}) - \Psi_n(c^n_{k}) }{c^n_{k+1}-c^n_{k}} (z - c^n_k)
% \]
The proposition is obviously verified. The next theorem  provides a performance guarantee for a solution returned by \eqref{prob:milp-IA}. 
\begin{theorem}\label{th:inner-milp-bound}
Suppose $(\overline{\btheta},\overline{\bz},\overline{\bx})$ be an optimal solution to the approximate problem \eqref{prob:milp-IA}, then
\begin{equation}\label{eq:th-bound-inner-milp}
|F(\overline{\bz}) - F^*|\leq  
\sum_{n\in [N]}\max_{z \in [L_n;U_n] } \left|\Psi_n(z) -\Gamma_n(z) \right| 
\end{equation}
where $F^*$ is the optimal value of \eqref{prob:ME-MCP-main}. 
\end{theorem}

Theorem \ref{th:inner-milp-bound} tells us that we can obtain an $(N\epsilon)$-approximation solution if we select piece-wise linear functions such that $\max_{n\in [N]}\max_{z \in [L_n,U_n] } \left|\Psi_n(z) -\Gamma_n(z) \right| \leq \epsilon$. It is clear that this can be always achievable for any $\epsilon>0$ by selecting sufficiently small intervals, because 
$$\lim_{\max_{n\in [N],k\in [K_n+1]} |c^n_{k+1}- c^n_{k}|\rightarrow 0}\max_{z \in [L_n;U_n] } |\Phi_n(z) -\Gamma_n(z)| = 0.$$
However, increasing the number of breakpoints also results in the growth of the size of the approximate MILP \eqref{prob:milp-IA}. Since we aim to optimize the size of \eqref{prob:milp-IA}, in the following, we demonstrate how to select the breakpoints in a manner that minimizes $K_n$ while ensuring an approximation guarantee.

\subsubsection{Optimizing the Number of Breakpoints}\label{subsec:concave-opt-points}
In this section, we explore an approach for minimizing the number of breakpoints while ensuring that the piece-wise linear approximation functions $\Gamma_n(z_n)$ remain within an $\epsilon$-neighborhood of the true objective functions $\Psi_n(z_n)$. To minimize the number of breakpoints, we would need to expand the gap between any consecutive breakpoints as much as possible, while guaranteeing that the approximation errors do not exceed a given threshold.  That is, from any breakpoint $a\in [L_n, U_{n}]$ and given $\epsilon>0$, we need to find a next breakpoint $b>a$ such that 
\[
\max_{z\in [a,b]}\left|\Psi_n(z) - \Gamma_n(z)\right| \leq \epsilon,
\]
recalling that
\[
\Gamma_n(z) = \left(\Psi_n(a) +\frac{\Psi_n(b) -\Psi_n(a) }{b-a} (z - a)\right), ~\forall z\in [a,b].
\]
Since we want to minimize the number of line segments, we will need to choose $b$ in such a way that  the gap $|b-a|$ is maximized. We then introduce the following problem to this end:
\begin{align}
    \max \left\{{b\in [a, U_n]}~\Big|~ \max_{z\in [a,b]}\left|\Psi_n(z) - \Gamma_n(z)\right|\leq \epsilon\right\} \label{eq:max-b}
\end{align}
Let us define, for ease of notation, denote:
\begin{align}
  \Lambda_n(t|a) &= \max_{z\in [a,t]}\left\{\Psi_n(z) - \Gamma_n(z)\right\}    \label{eq:lambda-t} \\
  \Theta_n(t)&= \frac{\Psi_n(t) -\Psi_n(a) }{t-a}.
\end{align}
For solving \eqref{eq:max-b}, we first introduce  the following lemma showing some important properties of the above functions:
\begin{lemma}
\label{lm:phi(t)}
The following results hold
\begin{itemize}
    \item [(i)] $\Theta_n(t)$ is (strictly) decreasing in $t$ 
   \item [(ii)] $\Lambda_n(t|a)$ can be computed by convex optimization 
   \item [(iii)] $\Lambda_n(t|a)$ is strictly monotonic increasing in $t$, for any $t\geq a$.
\end{itemize}

\end{lemma}

We now discuss how to solve  \eqref{eq:max-b} using the monotonicity and convexity of $\Theta_n(t)$ and  $\Lambda_n(t|a)$. We first write this problem as:
\[
\max \left\{t\in [a, U_n]\Bigg|~ \Lambda_n(t|a)\leq \epsilon\right\}.
\]
Since $\Lambda_n(t|a)$ is (strictly) increasing in $t$ and $\Lambda_n(a|a) = 0$, the above problem always yields a unique optimal solution that can be found by a binary search procedure.   Briefly, such a binary search can start with the  interval $[l,u]$ where $l=a$ and $u = U_n$. We then check if $\Lambda_n(u)\leq \epsilon$ then return $t^* = u$  as an optimal solution. Otherwise we take middle point $r = (u+l)/2$ and compute $\Lambda_n(r|a)$. If $\Lambda_n(r|a)<\epsilon$ we update the interval as $[r,u]$, otherwise we update the next  interval as $[l,u]$. This process stops when $u-l \leq \delta$ for a given threshold $\delta$. It is known that this procedure will terminate after  $\cO(\log(1/\delta))$ iterations.

Now, having an efficient method to solve \eqref{eq:max-b}, we describe below our method to (optimally) calculate breakpoints for the inner-approximation:
\begin{mdframed} [linewidth=1pt, roundcorner=5pt, backgroundcolor=gray!10]
\begin{itemize}
    \item(Step 1.) Let $c^n_1 = L_n$
    \item(Step 2.) For $k=1,\ldots$, compute the next point $c^n_{k+1}$ by solving 
    \[
     c^n_{k+1} = \text{argmax} \left\{t\in [c^n_{k}, U_n]\Bigg|~ \Lambda_n(t|c^n_{k})\leq \epsilon\right\}
    \]
    \item(Step 3.) Stop when $c^n_{k+1} = U_n$.
\end{itemize}
\end{mdframed}
We characterize the properties of the breakpoints returned by the above procedure in Theorem \ref{th:breakpoints} below (the proof is given in appendix):
\begin{theorem}\label{th:breakpoints}
The following properties hold:
\begin{itemize}
    \item [(i)] The numbers of breakpoints generated by the above procedure are optimal, i.e., for any set of breakpoints $\{c'_1,\ldots,c'_{K+1}\}$ such that $K<K_n$: 
    \[
    \max_{k\in [K]} \Lambda_n(c'_{k+1}|c'_{k}) > \epsilon,
    \]
     This implies that any inner piece-wise linear approximation of $\Psi_n(z)$ with a smaller number of breakpoints will yield an undesired approximation error. 
     \item [(ii)] The number of breakpoints $K_n+1$ can be bounded as 
     \[
     \frac{(U_n-L_n)\sqrt{L^\Psi_n}}{2\sqrt{\epsilon}} \leq K_n \leq \frac{(U_n-L_n)\sqrt{U^\Psi_n}}{\sqrt{2\epsilon}}
     \]
     where $L^\Psi_n$ and $U^\Psi_n$ are lower and upper bounds of $\Psi''_n(z_n)$ for $z_n\in [L_n,U_n]$, with a note that since $\Psi_n(z)$ is strictly concave in $z$, both $L^\Psi_n$ and $U^\Psi_n$ take positive values. 
\end{itemize}
\end{theorem}
Theorem \ref{th:breakpoints} establishes that the proposed procedure generates an optimal number of breakpoints. Specifically, there exists no other piecewise linear inner-approximation function with fewer breakpoints that achieves the same or smaller approximation gap compared to the one generated by the procedure. This result is intuitive, as the procedure optimizes each new breakpoint at every step. Consequently, for any smaller set of breakpoints, there will always be at least one pair of consecutive points where the approximation gap exceeds $\epsilon$.

The second part of Theorem \ref{th:breakpoints} highlights two important (and non-trivial) aspects. First, the breakpoint-finding procedure always terminates after a finite number of steps. Second, the number of steps (or generated breakpoints) is in $\mathcal{O}(1/\sqrt{\epsilon})$ and is generally proportional to the marginal value of the second-order derivative of $\Psi(z_n)$. This implies that the number of breakpoints increases to infinity as $\epsilon$ approaches zero. Moreover, the number of breakpoints will be larger if the concave function $\Psi(z_n)$ has high curvature and smaller if $\Psi(z_n)$ has low curvature (i.e., closer to a linear function).   In the special case where $\Psi(z_n)$ is linear, the upper and lower bounds satisfy $L^{\Psi}_n = U^{\Psi}_n = 0$, and only one breakpoint is needed ($K_n = 0$), which aligns with expectations.
%%%%%%%%%%%%%%%%%%%%%%%
%%%%%%%%%%%%%%%%%%%%%%%

\section{General Non-concave Market-expansion Function}\label{sec:general non-concave}

Our analysis thus far heavily relies on the assumption of concavity for the market expansion function. While such an assumption has been widely utilized in the literature and enables us to derive neat results (such as the concavity and submodularity of the objective function), aiding in efficiently solving the nonlinear optimization problem, it also presents certain limitations that may inaccurately capture market growth dynamics.

Specifically, the concavity assumption implies that the total demand of customer type $n$, calculated as $q_n g(u)$ (where $u$ represents the total expected customer utility offered by available facilities), grows rapidly when $u$ is small and gradually converges to $q_n$ as $u$ approaches infinity. However, this behavior may not be realistic as the addition of a few new facilities to the market would not immediately impact market growth. Conversely, it would be more realistic to assume that total demand grows slowly when $u$ is small and accelerates when a significant number of additional facilities are introduced to the market (resulting in a notable increase in $u$).
To further illustrate this remark, Figure \ref{fig:3market-expansion} below depicts the market growth behavior under two popular concave functions $g_1(t) = \frac{t}{t+\alpha}$  and $g_2(t) = 1-e^{-\alpha t}$ (as mentioned previously) and a non-concave function (i.e., sigmoidal function $g_3(t) =\frac{1}{1+e^{-\alpha t}}$). We can observe that both \( g_1(t) \) and \( g_2(t) \) grow rapidly as \( t \) increases from 0, slowing down only when \( t \) becomes sufficiently large. Mathematically, this is because both \( g_1(t) \) and \( g_2(t) \) are concave, resulting in decreasing gradients with respect to \( t \). In contrast, \( g_3(t) \) exhibits smaller growth rates when \( t \) is small and increases faster as \( t \) becomes larger. Consequently, \( g_3(t) \) would better reflect the influence of customer utility on the market size in practical scenarios. 

To address the aforementioned limitation of the concavity assumption, in this section, we will consider the ME-MCP with a \textbf{general non-concave market-expansion} function. We will first present a general method to approximate the ME-MCP via a MILP with arbitrary precision. We then demonstrate that, by identifying intervals where the objective function is either concave or convex, we can utilize the methods outlined earlier to optimally compute the breakpoints, thereby reducing the size of the MILP approximation. Furthermore, we will show that certain binary variables can be relaxed, further enhancing the efficiency  of the approximate MILP.

\begin{figure}[h!]
    \centering
    \includegraphics[width=\linewidth]{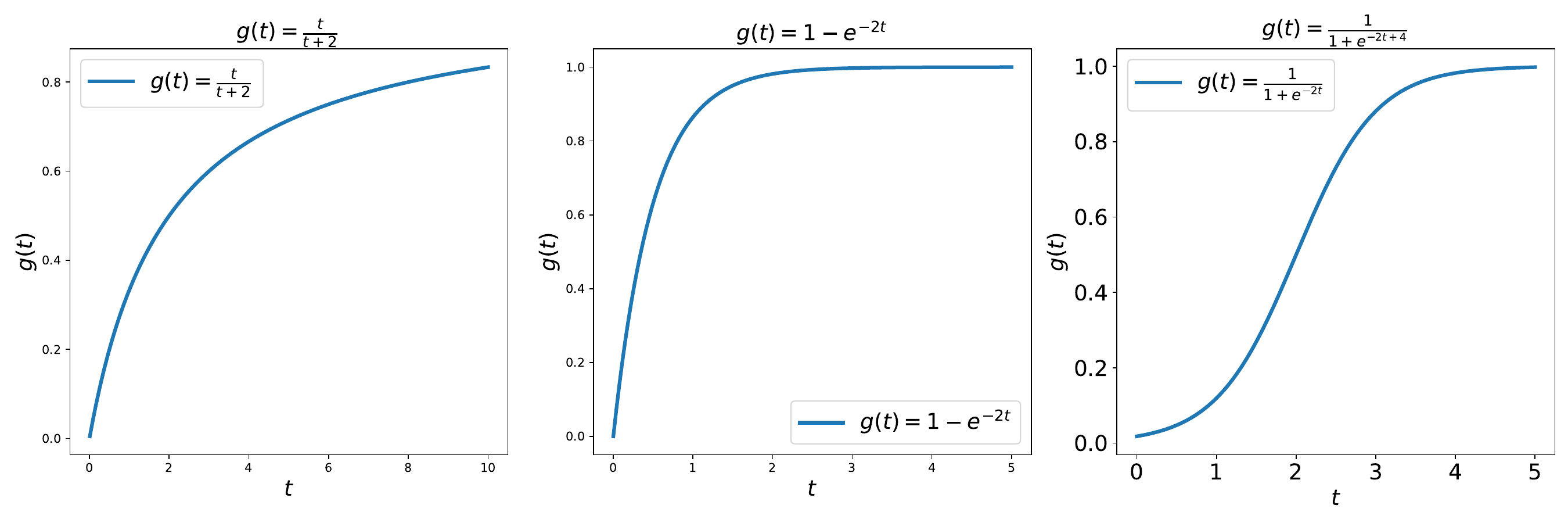}
    \caption{Plots of three types of market expansion functions}
    \label{fig:3market-expansion}
\end{figure}

\subsection{General MILP Approximation}
% https://www.desmos.com/calculator/zrpt6hrvca

We now consider the case that the market-expansion function $g(t)$ is not concave. As a result,  the objective function $\Phi_n(z_n)$ is no-longer concave in $z_n$.  We propose to approximate the non-concave function $\Phi_n(z_n)$ by a piece-wise linear function and show that \eqref{prob:main} can be approximated by an MILP with an arbitrary precision. 

First,  let us assume that  $g(t)$ is twice-differentiable in $z$, implying that $\Psi_n(z)$ is also twice-differentiable in $z$ for all $n\in [N]$. By taking the second derivative of this function and find solutions to $\Psi''_n(z) = 0$, one can identify intervals in which $\Psi_n(z)$ is either convex or concave. This allows us to well optimize the line segments and reduce the number of additional binary variables.  
That is, assume that we can split $[L_n;U_n]$ into some sub-intervals such that $\Psi_n(\cdot)$ is either concave or convex in each sub-interval. For each sub-interval, if $\Psi_n(z_n)$ is concave,we can use the method above to further split it into smaller interval $[c^n_k; c^n_{k+1}]$ in such a way that the gap between $\Psi_n(z)$ and the piece-wise linear function $\Gamma_n(z)$ is less than $\epsilon$ for any $z\in [c^n_k; c^n_{k+1}]$. On the other hand, if $\Psi_n(z_n)$ is concave, we show in Appendix  \ref{appdx:inner-convex} that one can use  methods similar to those described in Subsection \ref{subsec:concave-opt-points} to optimize the number of intervals  $[c^n_k; c^n_{k+1}]$.  We will describe this in detail later in the section. However, before that, let us show how to approximate the ME-MCP with a non-concave market-expansion function by using a MILP with such breakpoints.

Let us assume that after this procedure we also obtain a sequence of sub-intervals $\{c^n_1,\ldots,c^n_{K_n+1}\}$ such that within each interval $[c^n_k,c^n_{k+1}]$, $k\in [K_n]$, the gap between $\Psi_n(z_n)$ and the linear function $\Gamma_n(z)$, defined  as: 
\[
\Gamma_n(z)  = \Psi_n(c^n_k) + \frac{\Psi_n(c^n_{k+1}) - \Psi_n(c^n_{k})}{c^n_{k+1} - c^n_k} (z-c^n_k),
\]
is not larger than an $\epsilon >0$. We now can approximate $\Psi_n(z)$ via the following piece-wise linear function:
\begin{equation}\label{eq:Gamma-nonconcave}
    \Gamma_n(z)  =  \Psi_n(c^n_k) +\gamma^n_k (z - c^n_k),~\forall z\in [c^n_{k}; c^n_{k+1}], k\in [K_n].
\end{equation}
where 
\[
\gamma^n_k = \frac{\Psi_n(c^n_{k+1}) - \Psi_n(c^n_{k}) }{c^n_{k+1}-c^n_{k}},~\forall n\in [N],k\in [K_n].
\]
We now represent the condition \( z \in [c^n_k, c^n_{k+1}] \) using a binary variable \( y_{nk} \) and a continuous variable \( r_{nk} \). The binary variable \( y_{nk} \) satisfies \( y_{nk} \geq y_{n,{k+1}} \) for all \( k \in [K_n-1] \), and the continuous variable \( r_{nk} \) lies in the interval \([0,1]\) for all \( n \in [N] \) and \( k \in [K_n] \). Additionally, we require \( r_{nk} \geq y_{nk} \) and \( r_{n,k+1} \leq y_{nk} \) for all \( n \in [N] \) and \( k \in [K_n-1] \). This setup is to ensure that if \( y_{nk}=1 \), then \( r_{nk}=1 \); otherwise, if \( y_{nk}=0 \), then \( r_{nk'} = 0 \) for \( k' = k+1, \ldots \). The binary variables \( y_{nk} \) indicate the interval \([c^n_k, c^n_{k+1}]\) where \( z_n \) belongs, and the continuous variable \( r_{nk} \) captures the part \( z_n - c^n_k \). Using these variables, any \( z \in [L_n, U_n] \) can be expressed as \( z_n = \sum_{k \in [K_n-1]} (c^n_{k+1} - c^n_k) r_{nk} \). Moreover, the approximate function \( \Gamma_n(z) \) can be written as:
\[
    \Gamma_n(z) = \Psi_n(L_n) + \sum_{k \in [K_n-1]} \delta^n_k (c^n_{k+1} - c^n_k) r_{nk}
\]
We then can approximate  the ME-MCP by the following piece-wise linear problem:
\begin{align}
     \max_{\bx,\by,\bz,\br} &\left\{\sum_{n\in [N]}  \left(\Psi_n(z^L_n) + \sum_{k\in [K_n-1]} \gamma^n_k (c^n_{k+1} - c^n_k)r_{nk}\right)\right\} \label{prob:milp-2}\tag{\sf MILP-2} \\
    \mbox{subject to} 
    &\quad y_{nk}\geq y_{n,k+1},~k\in [K_n-1]\nonumber \\
&\quad  r_{nk}\geq y_{nk},~k\in [K_n-1]\nonumber \\
&\quad  r_{n,k+1}\leq y_{n,k},~k\in [K_n-1]\nonumber \\
    &\quad  \sum_{k\in [K_n-1]} (c^n_{k+1} - c^n_k)r_{nk} = \sum_{i\in [m]} x_i V_{ni} + 1,~ \forall n\in [N]\nonumber \\
    &\quad z_n = \sum_{k\in [K_n-1]} (c^n_{k+1} - c^n_k)r_{nk},~\forall n\in [N] \nonumber\\
    &\quad \sum_{i \in [m]} x_i \leq C \nonumber\\
    &\quad x_i, y_{nk}\in \{0,1\}, r_{nk},z_n\in [0,1],~\forall n\in [N],k\in [K_n] \nonumber.
\end{align}
This case differs from the concave market expansion scenario in that additional binary variables are required to construct the MILP approximation of the facility location problem. This raises concerns when a highly accurate approximation is needed, as the number of additional binary variables is proportional to the number of breakpoints used to form the piecewise linear function. In the subsequent section, we will demonstrate that some of these additional variables can be relaxed, resulting in a significantly simplified MILP approximation formulation.  

Before discussing this relaxation, we state the following theorem showing that solving \eqref{prob:milp-2} provides a solution with the same performance guarantees as solving \eqref{prob:milp-IA} in the concave market expansion case considered earlier.  

\begin{theorem}\label{th:non-concave-bound}
    Let \((\overline{\bx}, \overline{\by}, \overline{\bz}, \overline{\br})\) be an optimal solution to \eqref{prob:milp-2} and \((\bz^*, \bx^*)\) be optimal for the original MCP problem \eqref{prob:ME-MCP-main}. If the breakpoints are chosen such that \(|\Psi_n(z) - \Gamma_n(z)| \leq \epsilon\) for all \(n \in [N]\) and \(z \in [L_n, U_n]\), then \((\overline{\bx}, \overline{\bz})\) is feasible to \eqref{prob:ME-MCP-main} and
   $
    |F(\overline{\bz}) - F(\bz^*)| \leq N\epsilon.$

\end{theorem}

\subsection{Finding the Optimal Breakpoints}

As mentioned earlier, in the case of general market expansion functions, we can minimize the number of breakpoints by dividing the range $[L_n, U_n]$, for any $n \in [N]$, into sub-intervals where the objective function $\Psi_n(z_n)$ is either concave or convex in $z_n$. We can then apply the techniques described in Section \ref{subsec:concave-opt-points} (for concave intervals) and in Appendix \ref{appdx:inner-convex} (for convex ones) \footnote{This application is generally straightforward, as a convex function can be viewed as the inverse of a concave function.} The detailed steps are described as follows:

\begin{mdframed}[linewidth=1pt, roundcorner=5pt, backgroundcolor=gray!10]
\textbf{\underline{[Finding Optimal Breakpoints]}}

For any $n \in [N]$, set $a = L_n$ and select the first breakpoint $c^n_1 = L_n$:
\begin{itemize}
    \item \textbf{Step 1:} From $a$, find the nearest point $\delta > a$ such that $\Psi''_n(\delta) = 0$, and set $b = \min\{\delta, U_n\}$.
    \item \textbf{Step 2:} Within $[a, b]$:
    \begin{itemize}
        \item If the function $\Psi_n(z)$ is concave, use the methods described in Section [] to find the minimum number of breakpoints such that $|\Gamma_n(z) - \Psi_n(z)| \leq \epsilon$ for all $z \in [a, b]$.
        \item If $\Psi_n(z)$ is convex, employ a similar method described in Appendix [] to find the breakpoints such that $|\Gamma_n(z) - \Psi_n(z)| \leq \epsilon$ for all $z \in [a, b]$.
    \end{itemize}
    \item \textbf{Step 3:} If $b = U_n$, terminate the procedure. Otherwise, set $a = b$ and return to Step 1.
\end{itemize}
\end{mdframed}

From Theorem \ref{th:breakpoints} and Appendix \ref{appdx:inner-convex}, we can see that within any interval $[a, b]$ where $\Psi_n(z)$ is either concave or convex, the above procedure guarantees that the number of breakpoints is minimized. Moreover, Proposition \ref{prop:finding-breakpoint-nonconcave} below states that this procedure always terminates after a finite number of iterations and provides an upper bound on the number of breakpoints generated.

\begin{proposition}\label{prop:finding-breakpoint-nonconcave}
    The \textbf{[Finding Breakpoints]} procedure always terminates after a finite number of iterations (as long as there are a finite number of points $z \in [L_n, U_n]$ such that $\Psi''_n(z) = 0$). Moreover, the number of breakpoints $K_n$ can be bounded as:
    \[
    K_n \leq \frac{(U_n - L_n) \sqrt{U^\Psi_n}}{\sqrt{2\epsilon}} 
    \]
    where  $U^\Psi_n$ is an upper-bound of  $|\Psi^{''}_n(z)|$ in $[L_n,U_n]$. 
\end{proposition}

The proof can be found in the appendix where we leverage the second-order Taylor expansion to establish the bound. Similar to the case of concave market expansion, the number of breakpoints generated by the [\textbf{Finding Breakpoints}] procedure is always finite and bounded above by $\mathcal{O}(U^\Psi_n / \sqrt{\epsilon})$. Consequently, a higher number of breakpoints (and thus a larger MILP size) will be required if the desired accuracy $\epsilon$ is small or if the curvature of $\Psi_n(z)$ within \([L_n, U_n]\) is high. Conversely, fewer breakpoints will be needed if the curvature is low or the approximation accuracy requirement is less stringent.

% To represent the condition $z\in [c^n_{k}; c^n_{k+1}]$, we use  binary variables $y_{nk}\in \{0,1\}$ such that $y_{nk}\geq y_{n,k+1}$. We introduce the following MILP to approximate \eqref{prob:main}

% where $\delta^n_k$ is the slope of $\Psi_n(z)$ in $[c^n_{k}; c^n_{k+1}]$
% \[
% \delta^n_k = \frac{\Psi_n(c^n_{k+1}) - \Psi_n(c^n_{k}) }{c^n_{k+1}-c^n_{k}}
% \]

% $$g(t) = \frac{1}{1 + e^{-\alpha t + \beta }}$$

\subsection{Reducing the Number of Binary Variables.} 

As described in the previous section, the breakpoints are generated by dividing the interval $[L_n, U_n]$ (for any $n \in [N]$) into sub-intervals where $\Psi_n(z)$ is either concave or convex. The main problem \eqref{prob:ME-MCP-main} can then be approximated by \eqref{prob:milp-2}, whose size is proportional to the number of breakpoints. Specifically, \eqref{prob:milp-2} requires $\sum_{n} K_n$ additional binary variables. According to Proposition \ref{prop:finding-breakpoint-nonconcave}, the number of additional binary variables is proportional to $\frac{1}{\sqrt{\epsilon}}$, which increases as $\epsilon$ approaches zero. In the following, we show that the number of breakpoints (and thus the number of additional binary variables) can be significantly reduced by relaxing part of the additional binary variables.

Let define $\cK_n\subset [K_n]$ such that $\Psi_n(z)$ is concave in $[c^n_{k}; c^n_{k+1}]$ for all $k\in \cK_n$. We have the following theorem  stating that all the binary variables $y_{nk}$ for all $k\in \cK_n$ can be safely relaxed.

\begin{theorem}\label{th:non-concave-reducedMILP}
\eqref{prob:milp-2} is equivalent to 
\begin{align}
     \max_{\bx,\by,\bz,\br} &\left\{\sum_{n\in [N]}  \left(\Psi_n(L_n) + \sum_{k\in [K_n-1]} \gamma^n_k (c^n_{k+1} - c^n_k)r_{nk}\right)\right\} \label{prob:milp-relax}\tag{\sf MILP-3} \\
    \mbox{subject to} 
    &\quad y_{nk}\geq y_{n,k+1},~k\in [K_n-1]\label{ctr:ynk} \\
&\quad  r_{nk}\geq y_{nk},~k\in [K_n]\nonumber \\
&\quad  r_{n,k+1}\leq y_{n,k},~k\in [K_n-1]\nonumber \\
    &\quad  \sum_{k\in [K_n-1]} (c^n_{k+1} - c^n_k)r_{nk} = \sum_{i\in [m]} x_i V_{ni} + 1,~ \forall n\in [N]\label{ctr:r-x} \\
    &\quad z_n = \sum_{k\in [K_n-1]} (c^n_{k+1} - c^n_k)r_{nk},~\forall n\in [N] \nonumber\\
    &\quad \sum_{i \in [m]} x_i \leq C \nonumber\\
    &\quad x_i\in \{0,1\}, r_{nk}\in [0,1],~\forall n\in [N],k\in [K_n] \nonumber\\
    & \quad \mathbf{y_{nk} \in [0,1],~ \forall k\in \cK_n},\text{ and } y_{nk} \in  \{0,1\},~\forall k\in [K_n]\backslash \cK_n, ~\forall n\in [N]\nonumber.
\end{align}

\end{theorem}
The proof can be found in the appendix. Theorem \ref{th:non-concave-reducedMILP} indicates that some of the additional variables associated with regions where the function $\Psi_n(z_n)$ is concave can be relaxed. Specifically, if $\Psi_n(z_n)$ is concave across the entire interval \([L_n, U_n]\), all the additional binary variables can be relaxed, as in the case of the concave market expansion scenario discussed earlier.  

Since it is expected that the market expansion function \(g(t)\) will always increase and converge to 1 as \(z\) approaches infinity, \(\Psi_n(z_n)\) remains concave when \(z_n\) is sufficiently large. This allows a significant portion of the additional variables to be safely relaxed, thereby improving the overall computational efficiency of solving the MILP approximation.  

\section{Numerical Experiments}\label{sec:experiments}
This section presents the experimental results to assess the performance of three solutions methods introduced in Section \ref{sec:solutions_method}. In particular, the first Subsection~\ref{subsec:exp_settings} describes the benchmark datasets and experimental settings. The next Subsection~\ref{subsec:exp_epsilon} present a sensitivity analysis for choosing the error threshold $\epsilon$ in the Piece-wise Inner-approximation method.
Subsection~\ref{subsec:exp_concave} provides the computational results under the concave market expansion setting. Finally, Subsection~\ref{subsec:exp_nonconcave} presents the results on the general non-concave market expansion functions.

\subsection{Experiment Settings}
\label{subsec:exp_settings} 
We utilize three benchmark datasets in our experiments, all of which are widely used in  prior work in the context of competitive facility location \citep{Ljubic2018outer, mai2020multicut}.
\begin{itemize}
    \item \textbf{\texttt{HM14}}: there are $N$ customers and $m$ locations randomly located over a rectangular region. The number of customers $N$ takes values from $\{50, 100, 200, 400, 800\}$, while the number of locations $m$ varies over $\{25, 50, 100\}$, resulting in 15 combinations of ($N, m$).
    \item \textbf{\texttt{ORlib}}: this benchmark includes three types, namely \texttt{cap\_13} with four instances of $(N, m) = (50, 25)$, \texttt{cap\_13} with four instances of $(N, m) = (50, 50)$, and \texttt{cap\_abc} with three instances of $(N, m) = (1000, 100)$.
    \item \textbf{\texttt{P\&R-\texttt{NYC}}} (or \texttt{NYC} for short): this is a large test instance based on the park-and-ride facilities in New York City. The dataset is constituted by $N =$ 82,341 customers and $m = 59$ candidate locations.
\end{itemize}
For each test instance, the number of open locations $H$ is varied over $\{2, 3, \ldots, 10\}$. The utility associated with the customer zone $n$ and location $i$ is given by $v_{ni} = - \theta c_{ni}$ for $i\in[m]$ and $v^c_{ni} = - \gamma\theta c_{ni}$ for $i\in S^c$, where $S^c$ is randomly sampled from $[m]$ with $|S^c| = \ceil{m / 10}$, $\theta \in \{1, 5, 10\}$ and $\gamma \in \{0.01, 0.1, 1\}$ for the \texttt{\texttt{HM14}} and \texttt{OR
Lib} datasets, and $\theta \in \{0.5, 1, 2\}$ and $\gamma \in \{0.5, 1, 2\}$ for the \texttt{NYC}. Combining all the parameters results in total of 1215 test instances of the \texttt{\texttt{HM14}} dataset, 891 instances of the \texttt{ORLib} dataset, and 81 instances of the \texttt{NYC} dataset. For the objective function, we set the trade-off parameter $\lambda$ to 1.

For comparison, since there are no direct solution methods capable of solving the problem under consideration, we adapt state-of-the-art methods developed in the existing literature. Specifically, we include the following three approaches for comparison:

\begin{itemize}
    \item \textbf{Piece-wise Inner-Approximation (PIA)}: This is our method based on the piece-wise linear approximation described in Section~\ref{sec:solutions_method}. An important component of PIA is the parameter \(\epsilon\), which drives the accuracy of the approximate problem (or the guarantee of solutions provided by PIA). In these experiments, we select \(\epsilon = 0.01\), as this value is sufficiently small to offer almost optimal solutions for most cases (a detailed analysis is given in the next section).

    \item \textbf{Outer-Approximation (OA)}: This is an outer-approximation approach implemented in a cutting-plane manner, as described in Section~\ref{subsec:outer_approx}. This approach has been shown in previous work to achieve state-of-the-art performance for the competitive facility location problem without the market expansion and customer satisfaction terms \citep{mai2020multicut}. As supported by Theorem~\ref{thm:oa_exactness}, it can be seen that OA is an exact method under concave market-expansion functions but becomes heuristic for the general non-concave case.

    \item \textbf{Local Search (LS)}: This is a local search approach adapted from \citep{dam2022submodularity}. The approach is an iterative process consisting of three key steps: 
    \begin{enumerate}
        \item[(i)] A greedy step, where locations are selected one by one in a greedy manner,
        \item[(ii)] A gradient-based step, where gradient information is used to guide the search, and
        \item[(iii)] An exchanging step, where locations in the selected set are exchanged with ones outside to improve the objective values.
    \end{enumerate}
    Such a local search approach has been shown to achieve state-of-the-art performance for competitive facility location problems under general choice models \citep{dam2022submodularity,dam2023robust}. This approach, however, cannot guarantee achieving optimal solutions and is therefore considered heuristic. Nevertheless, as supported by the submodularity property shown in Section~\ref{sec:submodular}, LS can guarantee \((1 - 1/e)\)-approximation solutions.
\end{itemize}

%We test the performance of three solution methods presented in Section~\ref{sec:solutions_method}, including the Piece-wise Inner-approximation method (denoted as PIA), Outer-approximation method (denoted as OA), and the Local Search Heuristics method (denoted as LS). 
The experiments are implemented by C++ and run on Intel(R) Xeon(R) CPU E5-2698 v3 @ 2.30GHz. All linear programs are carried out by IBM ILOG CPLEX 22.1, with the time limit for each linear program being set to 5 hours. Each method under consideration (PIA, OA and LS) is given a time budget of 1 hours.

\subsection{Analysis for the Selection of  $\epsilon$}
\label{subsec:exp_epsilon}

\begin{figure}[htb]
    \centering
    % First subplot: Error figure
    \begin{subfigure}[b]{0.45\textwidth}
        \centering
        \includegraphics[width=\textwidth]{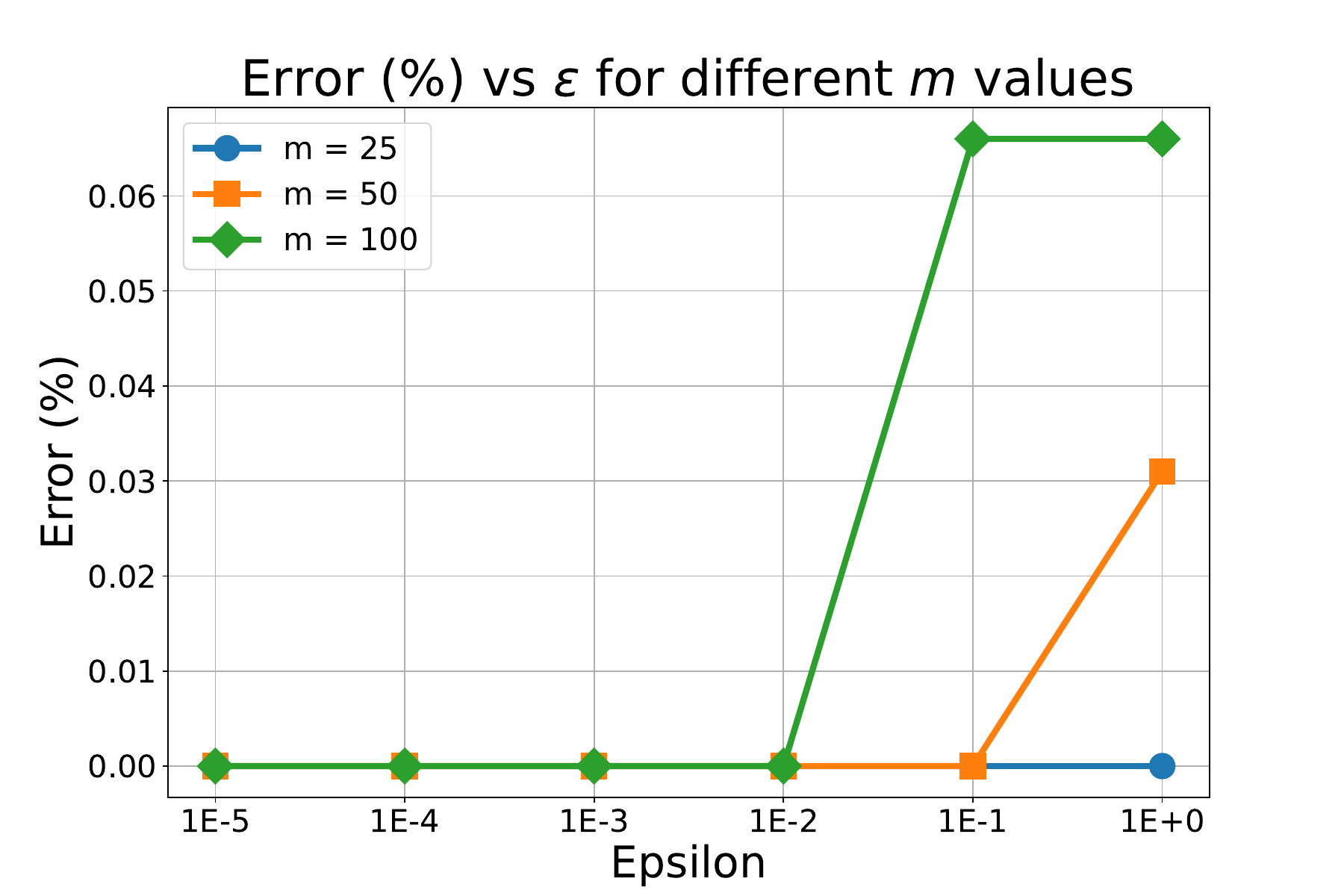}
        \caption{Error (\%) vs $\epsilon$}
        \label{fig:error_vs_epsilon}
    \end{subfigure}
    \hfill
    % Second subplot: Runtime figure
    \begin{subfigure}[b]{0.45\textwidth}
        \centering
        \includegraphics[width=\textwidth]{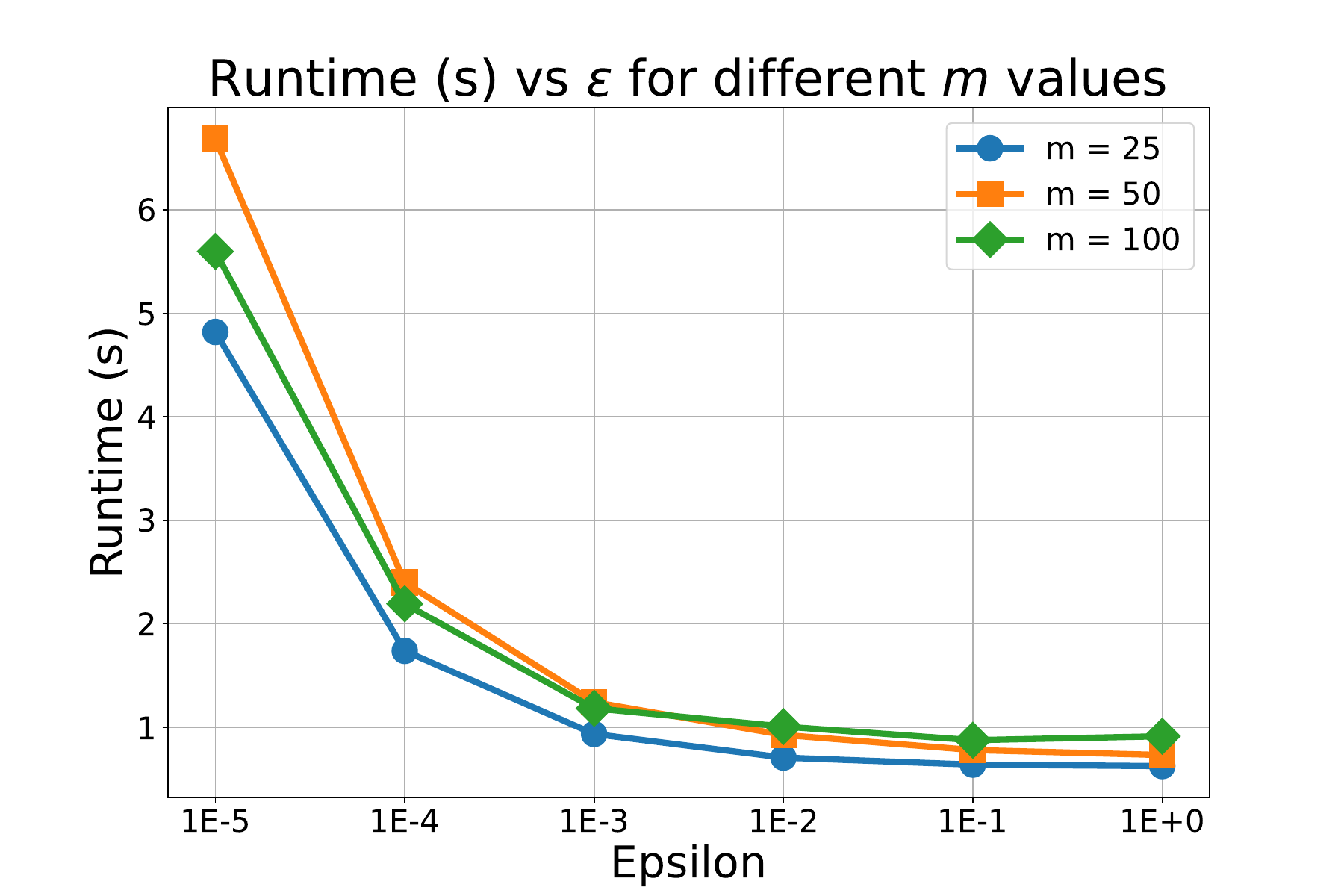}
        \caption{Runtime (s) vs $\epsilon$}
        \label{fig:runtime_vs_epsilon}
    \end{subfigure}

    \caption{Comparison of Error (\%) and Runtime (s) for different  $\epsilon$ values.}
    \label{fig:combined_results}
\end{figure}

We begin by conducting an experiment to analyze the practical impact of the parameter \(\epsilon\) on the performance of the PIA method. For this purpose, we select a concave market expansion function \(g(t) = 1 - e^{-t}\) and run the PIA method on three instances of the \texttt{HM14} dataset, where the number of customers is fixed at \(N = 50\) and \(m \in \{25, 50, 100\}\). The value of \(\epsilon\) is varied from \(1\text{E-5}\) to \(1.0\). For each value of \(\epsilon\), we measure and report the runtime and the percentage error of the corresponding solution relative to the solution obtained with \(\epsilon = 1\text{E-5}\). Here, we assume that setting \(\epsilon\) to \(1\text{E-5}\) will generally yield optimal solutions. The percentage errors  and runtimes are plotted in Figure \ref{fig:combined_results}.

For smaller values of \(\epsilon\) (e.g., \(1\text{E-5}\), \(1\text{E-4}\), and \(1\text{E-3}\)), the \textbf{Error (\%)} remains consistently zero for all problem sizes \(m\). This indicates that PIA achieves almost-optimal solutions when \(\epsilon\) is set to a sufficiently small value. However, this accuracy comes at the cost of increased computational time. For instance, when \(m = 50\), the runtime is \textbf{6.7 seconds} for \(\epsilon = 1\text{E-5}\) and reduces to \textbf{2.4 seconds} for \(\epsilon = 1\text{E-4}\), showing that even small increases in \(\epsilon\) can lead to significant improvements in efficiency.

As \(\epsilon\) increases to \(1\text{E-2}\), \(1\text{E-1}\), and \(1\text{E+0}\), a slight error begins to emerge, particularly for larger problem sizes. For example, at \(m = 50\), the error increases to \textbf{0.031\%} when \(\epsilon = 1\text{E-1}\). Similarly, at \(m = 100\), the error increases to \textbf{0.066\%} for both \(\epsilon = 1\text{E-1}\) and \(\epsilon = 1\text{E+0}\). Despite the presence of these small errors, the runtime decreases significantly. For \(m = 25\), the runtime reduces from \textbf{4.8 seconds} (\(\epsilon = 1\text{E-5}\)) to just \textbf{0.6 seconds} (\(\epsilon = 1\text{E+0}\)). This demonstrates that larger values of \(\epsilon\) lead to coarser approximations that accelerate computation but slightly compromise accuracy.

The results also reveal the scalability of the PIA method with respect to the problem size \(m\). As \(m\) increases from \(25\) to \(100\), the runtime increases, particularly for smaller values of \(\epsilon\). For instance, at \(\epsilon = 1\text{E-5}\), the runtime grows from \textbf{4.8 seconds} for \(m = 25\) to \textbf{6.69 seconds} for \(m = 50\). However, for larger values of \(\epsilon\), such as \(1\text{E-1}\) or \(1\text{E+0}\), the runtime remains relatively low even as \(m\) increases. This suggests that the computational burden of PIA can be effectively mitigated by selecting a larger \(\epsilon\) when slight errors are acceptable.

Overall, the results demonstrate that the choice of \(\epsilon\) is critical in balancing solution accuracy and computational efficiency. Smaller values of \(\epsilon\) are suitable for applications requiring high accuracy, as they ensure optimal solutions at the cost of longer runtimes. On the other hand, larger values of \(\epsilon\) significantly reduce runtime while maintaining near-optimal solutions, making them ideal for scenarios where computational speed is prioritized. Based on these analyses, we select \(\epsilon = 1\text{E-2}\) for our comparison results, as it appears to ensure an almost optimal solution for the PIA method while maintaining a reasonable size for the approximation problem.

\subsection{Concave Market Expansion}
\label{subsec:exp_concave}

\begin{table}[htb]
\centering
\caption{Comparison results for concave market-expansion.}
\label{table:concave_union}
\resizebox{0.9\textwidth}{!}{%
\begin{tabular}{@{}clllllccclccclccc@{}}
\toprule
\multirow{2}{*}{Dataset} &  & \multirow{2}{*}{$N$} &  & \multirow{2}{*}{$m$} &  & \multicolumn{3}{c}{\begin{tabular}[c]{@{}c@{}}No.   of \\      solved inst.\end{tabular}} &  & \multicolumn{3}{c}{\begin{tabular}[c]{@{}c@{}}No.   of \\      inst. with best obj.\end{tabular}} &  & \multicolumn{3}{c}{\begin{tabular}[c]{@{}c@{}}Average   \\      computing time (s)\end{tabular}} \\ \cmidrule(lr){7-9} \cmidrule(lr){11-13} \cmidrule(l){15-17} 
                         &  &                      &  &                      &  & PIA                          & OA                          & LS                           &  & PIA                             & OA                             & LS                             &  & PIA                            & OA                            & LS                              \\ \midrule
\texttt{HM14}                     &  & 50                   &  & 25                   &  & \textbf{81}                  & \textbf{81}                 & -                            &  & \textbf{81}                     & \textbf{81}                    & \textbf{81}                    &  & 0.17                           & 0.09                          & \textbf{0.05}                   \\
\texttt{HM14}                     &  & 50                   &  & 50                   &  & \textbf{81}                  & \textbf{81}                 & -                            &  & \textbf{81}                     & \textbf{81}                    & \textbf{81}                    &  & \textbf{0.03}                  & 0.13                          & 0.32                            \\
\texttt{HM14}                     &  & 50                   &  & 100                  &  & \textbf{81}                  & \textbf{81}                 & -                            &  & \textbf{81}                     & \textbf{81}                    & \textbf{81}                    &  & \textbf{0.04}                  & 0.06                          & 2.51                            \\
\texttt{HM14}                     &  & 100                  &  & 25                   &  & \textbf{81}                  & \textbf{81}                 & -                            &  & \textbf{81}                     & \textbf{81}                    & \textbf{81}                    &  & \textbf{0.04}                  & 0.23                          & 0.08                            \\
\texttt{HM14}                     &  & 100                  &  & 50                   &  & \textbf{81}                  & \textbf{81}                 & -                            &  & \textbf{81}                     & \textbf{81}                    & \textbf{81}                    &  & \textbf{0.04}                  & 0.07                          & 0.60                            \\
\texttt{HM14}                     &  & 100                  &  & 100                  &  & \textbf{81}                  & \textbf{81}                 & -                            &  & \textbf{81}                     & \textbf{81}                    & \textbf{81}                    &  & \textbf{0.07}                  & 0.12                          & 4.99                            \\
\texttt{HM14}                     &  & 200                  &  & 25                   &  & \textbf{81}                  & \textbf{81}                 & -                            &  & \textbf{81}                     & \textbf{81}                    & \textbf{81}                    &  & \textbf{0.05}                  & 0.06                          & 0.14                            \\
\texttt{HM14}                     &  & 200                  &  & 50                   &  & \textbf{81}                  & \textbf{81}                 & -                            &  & \textbf{81}                     & \textbf{81}                    & \textbf{81}                    &  & \textbf{0.09}                  & 0.13                          & 1.18                            \\
\texttt{HM14}                     &  & 200                  &  & 100                  &  & \textbf{81}                  & \textbf{81}                 & -                            &  & \textbf{81}                     & \textbf{81}                    & \textbf{81}                    &  & \textbf{0.17}                  & 0.35                          & 10.00                           \\
\texttt{HM14}                     &  & 400                  &  & 25                   &  & \textbf{81}                  & \textbf{81}                 & -                            &  & \textbf{81}                     & \textbf{81}                    & \textbf{81}                    &  & \textbf{0.12}                  & 0.20                          & 0.26                            \\
\texttt{HM14}                     &  & 400                  &  & 50                   &  & \textbf{81}                  & \textbf{81}                 & -                            &  & \textbf{81}                     & \textbf{81}                    & \textbf{81}                    &  & \textbf{0.22}                  & 0.37                          & 2.33                            \\
\texttt{HM14}                     &  & 400                  &  & 100                  &  & \textbf{81}                  & \textbf{81}                 & -                            &  & \textbf{81}                     & \textbf{81}                    & \textbf{81}                    &  & \textbf{0.49}                  & 1.31                          & 20.62                           \\
\texttt{HM14}                     &  & 800                  &  & 25                   &  & \textbf{81}                  & \textbf{81}                 & -                            &  & \textbf{81}                     & \textbf{81}                    & \textbf{81}                    &  & \textbf{0.29}                  & 0.59                          & 0.50                            \\
\texttt{HM14}                     &  & 800                  &  & 50                   &  & \textbf{81}                  & \textbf{81}                 & -                            &  & \textbf{81}                     & \textbf{81}                    & \textbf{81}                    &  & \textbf{0.91}                  & 1.33                          & 4.60                            \\
\texttt{HM14}                     &  & 800                  &  & 100                  &  & \textbf{81}                  & \textbf{81}                 & -                            &  & \textbf{81}                     & \textbf{81}                    & \textbf{81}                    &  & \textbf{1.54}                  & 3.04                          & 41.42                           \\ \midrule
\texttt{cap\_10}                  &  & 50                   &  & 25                   &  & \textbf{324}                 & \textbf{324}                & -                            &  & \textbf{324}                    & \textbf{324}                   & \textbf{324}                   &  & 0.26                           & 0.11                          & \textbf{0.05}                   \\
\texttt{cap\_13}                  &  & 50                   &  & 50                   &  & \textbf{324}                 & \textbf{324}                & -                            &  & \textbf{324}                    & \textbf{324}                   & \textbf{324}                   &  & 0.76                           & \textbf{0.16}                 & 0.32                            \\
\texttt{cap\_abc}                 &  & 1000                 &  & 100                  &  & \textbf{243}                 & \textbf{243}                & -                            &  & \textbf{243}                    & \textbf{243}                   & \textbf{243}                   &  & 13.35                          & \textbf{7.40}                 & 54.27                           \\ \midrule
\texttt{NYC}                      &  & 82341                &  & 59                   &  & \textbf{81}                  & 76                          & -                            &  & \textbf{81}                     & \textbf{81}                    & \textbf{81}                    &  & 1433.80                        & 5547.00                       & \textbf{973.86}                 \\ \midrule
\multicolumn{5}{c}{Total}                                                    &  & \textbf{2187}                & 2182                        & -                            &  & \textbf{2187}                   & \textbf{2187}                  & \textbf{2187}                  &  & -                              & -                             & -                               \\ \bottomrule
\end{tabular}}
\end{table}

In this section, we present the numerical results obtained by three solution methods for addressing the facility location problem with concave market expansion. The market expansion function is selected as \( g(t) = 1 - e^{-\alpha t} \) with \(\alpha = 1\), a popular choice in prior studies related to market expansion in competitive facility location problems \citep{ABOOLIAN2007a,lin2022locating}. The results for three datasets are reported in Table~\ref{table:concave_union}, where each row contains results for instances grouped by \((N, m)\).

Three evaluation criteria are considered: (1) the number of instances solved to optimality within the time budget, (2) the number of instances where the corresponding method achieves the best solution among the three methods, and (3) the average computing time in seconds required to confirm the optimality of the solution. In this setting, since the objective function is concave (Theorem~\ref{th:concavity}), both PIA and OA serve as exact (or near-exact) methods, while LS remains heuristic. Therefore, the number of solved instances is only reported for PIA and OA.

The results generally show that PIA emerges as the most efficient method, consistently solving all instances across all datasets and configurations. This reliability is evident in both the \texttt{HM14} and \texttt{cap} datasets, where PIA solves all 81 and 324 instances, respectively, achieving the best objective value in every case. Furthermore, its computational efficiency is particularly noteworthy, especially for larger datasets such as \texttt{cap\_abc}, where PIA completes the task in \textbf{13.35} seconds. This combination of reliability, optimality, and efficiency positions PIA as the most favorable method for solving these optimization problems.

OA closely mirrors the performance of PIA in terms of solution quality and reliability. It also solves all instances across the datasets and achieves the best objective value in every case. However, OA tends to require slightly higher computational times, particularly for larger datasets. For example, in the \texttt{NYC} dataset, OA’s computing time (\textbf{5547.00} seconds) is significantly higher than that of PIA (\textbf{1433.80} seconds). Despite this, OA remains a viable choice for scenarios where computational cost is less of a concern, given its ability to consistently deliver high-quality solutions. This observation aligns with the fact that OA has been recognized as a state-of-the-art approach for competitive facility location problems under fixed market sizes \citep{mai2020multicut}.

LS, on the other hand, provides a contrasting performance profile. While LS often requires less computational time compared to PIA and OA, as demonstrated in the \texttt{NYC} dataset (\textbf{973.86} seconds), it does not guarantee optimal or near-optimal solutions. This limitation is reflected in the ``-'' entries under the ``Number of solved instances'' column for LS. These entries highlight that LS, being a heuristic method, prioritizes computational speed over solution quality. Although LS can occasionally match the best objective values achieved by PIA and OA, such occurrences are less consistent. As a result, LS is less suitable for applications where solution quality or optimality is critical.

The scalability of PIA and OA across increasing problem sizes further underscores their suitability for large-scale instances. As the values of \(N\) and \(m\) grow, both methods maintain their ability to solve all instances while achieving the best objective values. In contrast, LS, despite its computational efficiency, struggles to balance scalability and solution quality, particularly in larger datasets.

In summary, the results highlight that PIA stands out as the most reliable and efficient method, particularly for scenarios requiring optimal solutions. OA offers a strong alternative, especially for smaller datasets, though it may incur higher computational costs for larger problems. LS, with its emphasis on computational speed, is best suited for applications where solution quality is less critical, and computational resources are limited.

\subsection{General Non-concave Market Expansion}
\label{subsec:exp_nonconcave}

\begin{table}[htb]
\centering
\caption{Comparison results for non-concave market expansion.}
\label{table:nonconcave_union}
\resizebox{0.9\textwidth}{!}{%
\begin{tabular}{@{}clllllccclccclccc@{}}
\toprule
\multirow{2}{*}{Dataset} &  & \multirow{2}{*}{$N$} &  & \multirow{2}{*}{$m$} &  & \multicolumn{3}{c}{\begin{tabular}[c]{@{}c@{}}No.   of \\      solved inst.\end{tabular}} &  & \multicolumn{3}{c}{\begin{tabular}[c]{@{}c@{}}No.   of \\      inst. with best obj.\end{tabular}} &  & \multicolumn{3}{c}{\begin{tabular}[c]{@{}c@{}}Average   \\      computing time (s)\end{tabular}} \\ \cmidrule(lr){7-9} \cmidrule(lr){11-13} \cmidrule(l){15-17} 
                         &  &                      &  &                      &  & PIA                          & OA                          & LS                           &  & PIA                              & OA                            & LS                             &  & PIA                                & OA                               & LS                       \\ \midrule
\texttt{HM14}                     &  & 50                   &  & 25                   &  & \textbf{81}                  & -                           & -                            &  & \textbf{81}                      & \textbf{81}                   & \textbf{81}                    &  & \textbf{0.04}                      & 0.21                             & 0.06                     \\
\texttt{HM14}                     &  & 50                   &  & 50                   &  & \textbf{81}                  & -                           & -                            &  & \textbf{81}                      & \textbf{81}                   & \textbf{81}                    &  & \textbf{0.03}                      & 0.08                             & 0.32                     \\
\texttt{HM14}                     &  & 50                   &  & 100                  &  & \textbf{81}                  & -                           & -                            &  & \textbf{81}                      & \textbf{81}                   & \textbf{81}                    &  & 0.08                               & \textbf{0.05}                    & 2.53                     \\
\texttt{HM14}                     &  & 100                  &  & 25                   &  & \textbf{81}                  & -                           & -                            &  & \textbf{81}                      & \textbf{81}                   & \textbf{81}                    &  & \textbf{0.03}                      & 0.20                             & 0.09                     \\
\texttt{HM14}                     &  & 100                  &  & 50                   &  & \textbf{81}                  & -                           & -                            &  & \textbf{81}                      & \textbf{81}                   & \textbf{81}                    &  & \textbf{0.07}                      & \textbf{0.07}                    & 0.61                     \\
\texttt{HM14}                     &  & 100                  &  & 100                  &  & \textbf{81}                  & -                           & -                            &  & \textbf{81}                      & \textbf{81}                   & \textbf{81}                    &  & \textbf{0.04}                      & 0.11                             & 5.04                     \\
\texttt{HM14}                     &  & 200                  &  & 25                   &  & \textbf{81}                  & -                           & -                            &  & \textbf{81}                      & \textbf{81}                   & \textbf{81}                    &  & \textbf{0.06}                      & \textbf{0.06}                    & 0.15                     \\
\texttt{HM14}                     &  & 200                  &  & 50                   &  & \textbf{81}                  & -                           & -                            &  & \textbf{81}                      & \textbf{81}                   & \textbf{81}                    &  & \textbf{0.06}                      & 0.11                             & 1.20                     \\
\texttt{HM14}                     &  & 200                  &  & 100                  &  & \textbf{81}                  & -                           & -                            &  & \textbf{81}                      & \textbf{81}                   & \textbf{81}                    &  & \textbf{0.09}                      & 0.28                             & 10.07                    \\
\texttt{HM14}                     &  & 400                  &  & 25                   &  & \textbf{81}                  & -                           & -                            &  & \textbf{81}                      & \textbf{81}                   & \textbf{81}                    &  & \textbf{0.08}                      & 0.19                             & 0.26                     \\
\texttt{HM14}                     &  & 400                  &  & 50                   &  & \textbf{81}                  & -                           & -                            &  & \textbf{81}                      & \textbf{81}                   & \textbf{81}                    &  & \textbf{0.12}                      & 0.34                             & 2.36                     \\
\texttt{HM14}                     &  & 400                  &  & 100                  &  & \textbf{81}                  & -                           & -                            &  & \textbf{81}                      & \textbf{81}                   & \textbf{81}                    &  & \textbf{0.19}                      & 0.97                             & 20.50                    \\
\texttt{HM14}                     &  & 800                  &  & 25                   &  & \textbf{81}                  & -                           & -                            &  & \textbf{81}                      & \textbf{81}                   & \textbf{81}                    &  & \textbf{0.18}                      & 0.58                             & 0.50                     \\
\texttt{HM14}                     &  & 800                  &  & 50                   &  & \textbf{81}                  & -                           & -                            &  & \textbf{81}                      & \textbf{81}                   & \textbf{81}                    &  & \textbf{0.28}                      & 1.28                             & 4.72                     \\
\texttt{HM14}                     &  & 800                  &  & 100                  &  & \textbf{81}                  & -                           & -                            &  & \textbf{81}                      & \textbf{81}                   & \textbf{81}                    &  & \textbf{0.46}                      & 2.68                             & 43.18                    \\ \midrule
\texttt{cap\_10}                  &  & 50                   &  & 25                   &  & \textbf{324}                 & -                           & -                            &  & \textbf{324}                     & 268                           & 308                            &  & 0.57                               & \textbf{0.01}                    & 0.06                     \\
\texttt{cap\_13}                  &  & 50                   &  & 50                   &  & \textbf{324}                 & -                           & -                            &  & \textbf{324}                     & 288                           & 276                            &  & 0.65                               & \textbf{0.01}                    & 0.32                     \\
\texttt{cap\_abc}                 &  & 1000                 &  & 100                  &  & \textbf{222}                          & -                           & -                            &  & 240                              & 241                           & \textbf{242}                   &  & 36.43                              & \textbf{1.56}                    & 57.17                    \\ \midrule
\texttt{NYC}                      &  & 82341                &  & 59                   &  & \textbf{81}                  & -                           & -                            &  & \textbf{81}                      & \textbf{81}                   & \textbf{81}                    &  & \textbf{715.62}                    & 4258.47                          & 963.25                   \\ \midrule
\multicolumn{5}{c}{Total}                                                    &  & 2166                         & -                           & -                            &  & \textbf{2184}                    & 2093                          & 2122                           &  & -                                  & -                                & -                        \\ \bottomrule
\end{tabular}}
\end{table}

In this experiment, we evaluate the performance of PIA under a general non-concave market expansion function. The market expansion function is defined as \( g(t) = \frac{1}{1 + e^{-\alpha(t - \beta)}} \), where \(\alpha = 5\) and \(\beta = 4\). The results are presented in Table~\ref{table:nonconcave_union}, using the same format as in the previous experiment. Since both OA and LS are heuristic methods in this setting, we report the number of solved instances only for PIA.

Similar to the previous experiment, PIA emerges as the most efficient method. It consistently solves all instances across the different datasets and configurations, as reflected in the ``No. of solved instances'' column. Unlike OA and LS, PIA guarantees optimal or near-optimal solutions. This highlights its ability to handle the complexity of the solution space, particularly in cases where other methods fail. For example, in the \texttt{HM14} and \texttt{cap} datasets, PIA solves all instances while achieving the best objective values. This makes PIA the  preferred choice for problems requiring both practical accuracy and reliability.

The analysis of computing times provides further insights into the trade-offs between solution quality and efficiency. While ensuring optimal or near-optimal solutions, PIA maintains competitive computing times across all problem sizes. For instance, in the \texttt{cap\_abc} dataset with \(N = 1000, m = 100\), PIA completes the task in \textbf{36.43} seconds, which is slower than OA (\textbf{1.56} seconds) but significantly faster than LS (\textbf{57.17} seconds). OA often demonstrates shorter computational times, particularly for smaller datasets, but this efficiency comes at the cost of reduced robustness. Notably, the unusually fast runtime of OA in the second dataset coincides with its poor solution quality, which can be attributed to invalid cutting planes introduced at the early stages of the algorithm. LS, on the other hand, achieves the fastest computing times in some cases, such as the \texttt{NYC} dataset, but its inability to guarantee solution quality undermines its overall performance.

Scalability is another critical factor. PIA demonstrates strong scalability as the problem size increases, maintaining its ability to solve all instances even for large datasets. For example, in the \texttt{NYC} dataset (\(N = 82341\)), PIA successfully solves all instances, achieving the best objective values while maintaining a reasonable computational cost. In contrast, OA and LS struggle to scale effectively, with performance deteriorating as the problem size increases. This issue is particularly pronounced in the larger datasets, such as \texttt{cap\_abc} and \texttt{NYC}, where neither OA nor LS matches the robustness and reliability observed in PIA.

The table also highlights interesting results regarding solution quality. For the \texttt{HM14} and \texttt{NYC} datasets, all methods achieve comparable solution quality; however, PIA stands out as the fastest method in these instances. In the second dataset, ORlib, PIA demonstrates superior performance by providing the best solutions for all 324 test instances in the \texttt{cap\_10} and \texttt{cap\_13} cases. In contrast, OA solves only 268 and 288 instances, while LS solves 308 and 276 instances, respectively. For the large \texttt{cap\_abc} dataset within ORlib, PIA solves 222 out of 324 instances to optimality. Despite this limitation, the number of best solutions found by PIA remains comparable to those obtained by OA and LS, further reinforcing its overall reliability and efficiency.

In summary, the results clearly demonstrate that PIA is the most effective method for solving the facility location problem under general non-concave market expansion functions. PIA guarantees near-optimal solutions while maintaining competitive computational efficiency and strong scalability. While OA and LS offer faster runtimes in specific cases, their inability to consistently solve instances and ensure solution quality limits their applicability. For problems requiring reliability, accuracy, and scalability, PIA remains the method of choice.

% The performance of methods is influenced by the shape of the objective function over the feasible region, which in turns is influenced by the shape of function $g(t)$.

\subsection{Impact of the Slope of the Market Expansion Function}

\begin{figure}[htb]
    \centering
    % First subplot: cap10 dataset
    \begin{subfigure}[b]{0.49\textwidth}
        \centering
        \includegraphics[width=\textwidth]{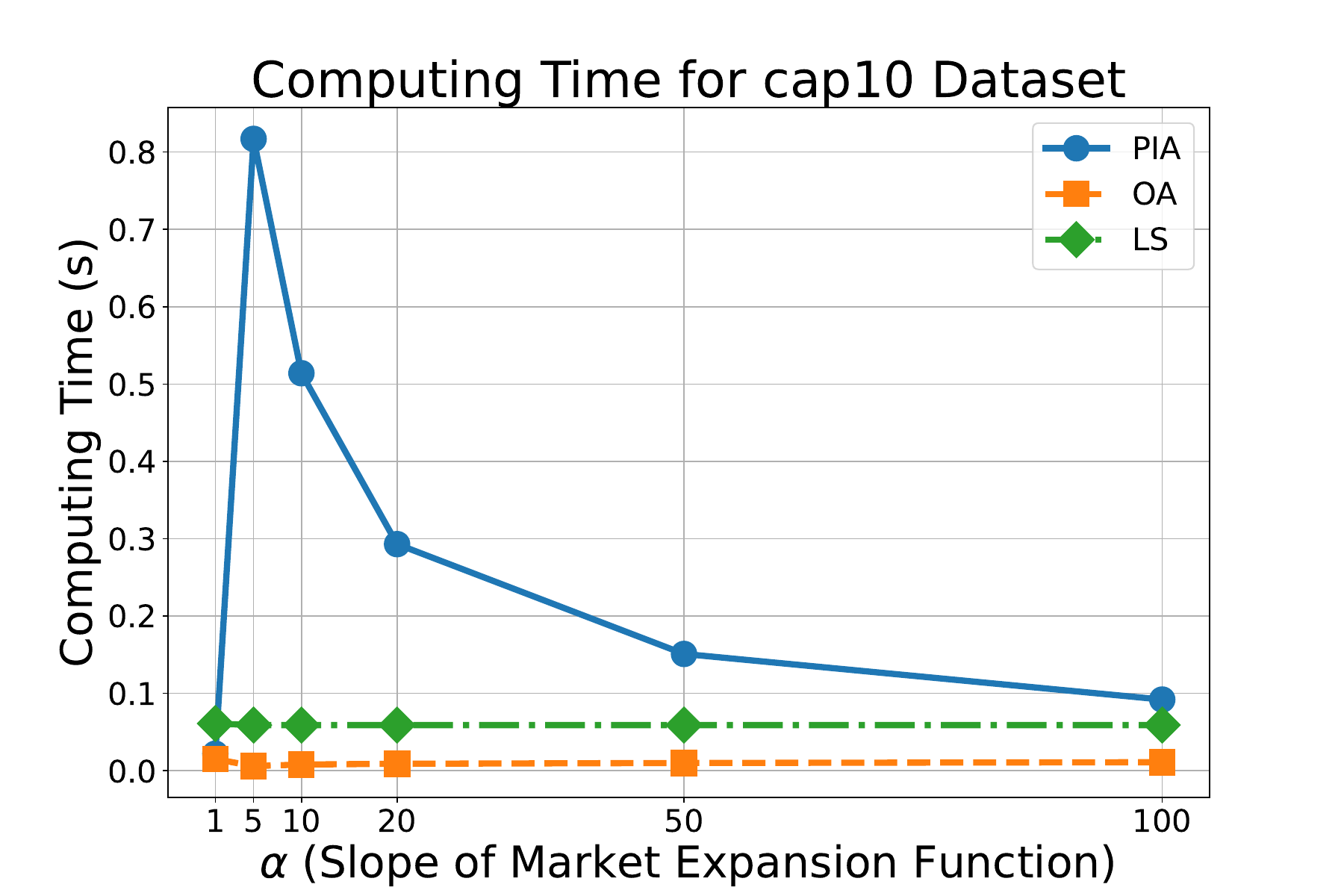}
        \caption{Computing Time for cap10 Dataset}
        \label{fig:computing_time_cap10}
    \end{subfigure}
    \hfill
    % Second subplot: cap13 dataset
    \begin{subfigure}[b]{0.49\textwidth}
        \centering
        \includegraphics[width=\textwidth]{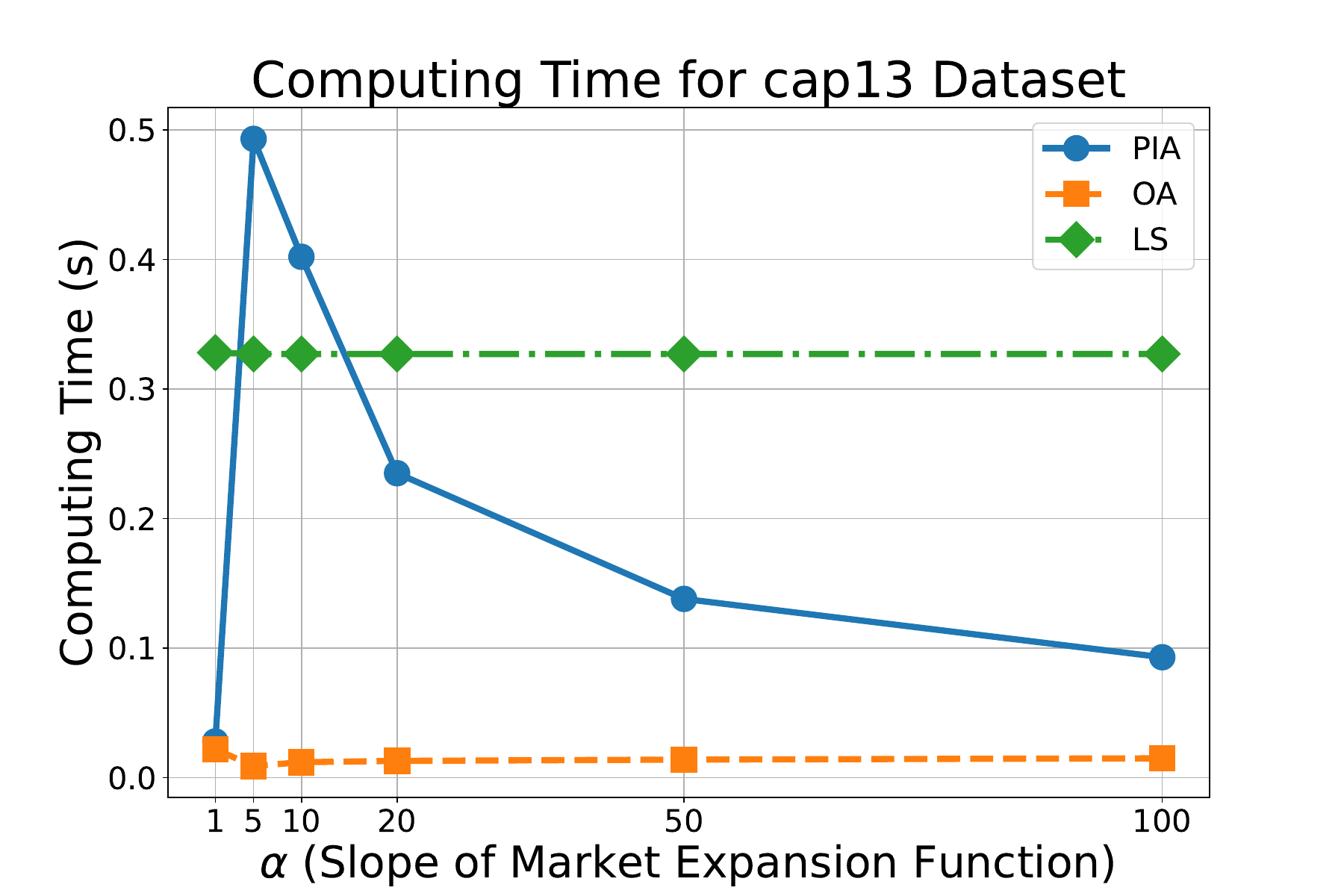}
        \caption{Computing Time for cap13 Dataset}
        \label{fig:computing_time_cap13}
    \end{subfigure}
    
    \caption{Comparison of Computing Times for cap10 and cap13 Datasets Across Different $\alpha$ Values.}
    \label{fig:computing_time_comparison}
\end{figure}

The slope of the market expansion function reflects how the market grows as the total consumer surplus increases. In the context of concave market expansion with \( g(t) = 1 - e^{-\alpha t} \), this behavior is captured by the parameter \(\alpha\). Since \(g(t)\) is an increasing function of \(\alpha\), and the second-order derivative of \(g(t)\) with respect to \(t\) is given by \( g''(t) = \alpha^2 e^{-\alpha t} \), it decreases exponentially to zero as \(\alpha\) increases. Intuitively, when \(\alpha\) is large, the market expansion function increases more rapidly towards 1 and exhibits lower curvature. On the other hand, when \(\alpha\) is smaller, the market expansion function adheres to a higher curvature.  Moreover, the bounds reported in Theorem~\ref{th:breakpoints} indicate that functions with lower curvature require fewer breakpoints in the PIA, and vice versa. Consequently, a higher \(\alpha\) leads to a lower curvature of the objective function, resulting in a smaller approximation problem. As a result, the PIA method is expected to run faster when \(\alpha\) is larger. To experimentally illustrate this, we conduct a series of experiments with varying values of \(\alpha\).  

To this end, we choose the concave market expansion function \( g(t) = 1 - e^{-\alpha t} \) and vary the parameter \(\alpha \in \{1, 5, 10, 20, 50, 100\}\). We select the ORlib dataset (excluding \texttt{cap\_abc} due to its large size) for this experiment since it is most sensitive to the concavity of \(g(t)\). The results are plotted in Figure \ref{fig:computing_time_comparison}. As expected, when \(\alpha\) is small, PIA requires more computational time, whereas it becomes faster as \(\alpha\) increases. This aligns well with the intuition discussed earlier: larger values of \(\alpha\) reduce the curvature of the objective function, thereby requiring fewer breakpoints for the PIA method. In contrast, the runtimes of OA and LS are not significantly affected by changes in \(\alpha\) -- their overall runtimes remain stable when \(\alpha\) increases.

%Explanation for the result in Table \ref{table:market_expansion}: With small $\alpha$, the objective function is almost concave and flat, resulting in small number of breakpoints. With large $\alpha$, the expansion happens steadily (high curvature), but in turn the convex segment is short, resulting in small number of binary variables...

%In Table \ref{table:market_expansion}, we evaluate the impacts of the shape of the market expansion function on the performance of algorithms. 
%\input{tables/market_expansion}

% In Table X, we evaluate the impacts of the competitor utility on the market expansion by varying the parameter $\gamma \in \{\}$ with $\theta$ fixed to 1.

\section{Conclusion}\label{sec:concl}
\label{sec:concl}
In this paper, we studied a competitive facility location problem with market expansion and a customer-centric objective, aiming to capture the dynamics of the market while improving overall customer satisfaction. The novel problem formulation, to the best of our knowledge, cannot be directly solved to near-optimality by any existing approach, particularly under a general non-concave market expansion model.  

To address these challenges, we first demonstrated that under concave market expansion, the objective function exhibits both concavity and submodularity. This allows the problem to be solved exactly using an outer-approximation approach. However, this property does not hold under a general non-concave market expansion function. To overcome this limitation, we proposed a new approach based on an inner-approximation method. We showed that our PIA approach consistently yields smaller approximation gaps compared to any outer-approximation counterpart. Furthermore, the inner-approximation program, in addition to being able to achieve arbitrarily precise solutions, can be formulated as a MILP.  

We further strengthened the proposed approach by developing an optimal strategy for selecting breakpoints in the PIA, minimizing the size of the approximation problem for a given precision level. Additionally, we showed how to significantly reduce the number of binary variables in the case of non-concave market expansion by examining regions of the objective function where it behaves either as convex or concave.  

Experiments conducted for both concave and non-concave market expansion settings demonstrate the efficiency of the proposed PIA approach in terms of solution quality, solution guarantees, and runtime performance. We also provided an analysis of the impact of the approximation accuracy threshold \(\epsilon\) and the slope parameter \(\alpha\) of the market expansion function on the performance of PIA.  

Future research will focus on developing an advanced version of PIA that returns exact solutions or extending the proposed PIA approach to other variants of the competitive facility location problem. For instance, it could be applied to models involving more complex choice behaviors, such as the nested logit or multi-level nested logit models \citep{train2009discrete,mai2017dynamic}.

\bibliographystyle{plainnat_custom}
\bibliography{refs-2}

\pagebreak

\appendix

\begin{center}
    {\Huge Appendix}
\end{center}

\section{Missing Proofs}
\subsection{Proof of Theorem \ref{th:concavity}}

\begin{proof}
Since $\log(z_n)$ is concave in $z_n$, we only consider the first term of $\Psi_n(z_n)$. Let us denote
\[
\cH(z) = g(\log(z))   \left( \frac{ z - U^c_n}{z}\right).
\]
We simply taking  the first and second-order derivatives of $\cH(z)$ to have
\begin{align}
    \cH'_n(z) &=  \frac{g'(\log z)}{z} - \frac{U^c_n g'(\log z)}{z^2} +  \frac{U^c_ng(\log z)}{z^2} 
    \nonumber\\
    \cH''(z) &= \frac{g''(\log z)}{z^2} -  \frac{g'(\log z)}{z^2} - \frac{U^c_n g''(\log z)}{z^3} + \frac{2U^c_n g'(\log z)}{z^3} + \frac{U^c_n g'(\log z)}{z^3} - \frac{2U^c_ng(\log z)}{z^3}\nonumber \\
    &=g''(\log z)\frac{1}{z^2}\left(1-\frac{U^c_n}{z}\right) + g'(\log z)\frac{1}{z^2}\left(-1 + \frac{3U^c_n}{z}\right)  - \frac{2U^c_ng(\log z)}{z^3}
    \nonumber
\end{align}
We now see that  $1-U^c_n/z\geq 0$ and $g''(\log z)\leq 0$ (because $g(t)$ is concave), thus 
\[
 g''(\log z)\frac{1}{z^2}\left(1-\frac{U^c_n}{z}\right) \leq 0.
\]
Moreover, since $g'(\log z)\geq 0$ and $U^c_n \leq z$, we have
\begin{align}
    \frac{3U^c_ng'(\log z)}{z} -{g'(\log z)} - \frac{2U^c_n g(\log z)}{z} &\leq \frac{3U^c_ng'(\log z)}{z} -\frac{U^c_n g'(\log z)}{z} - \frac{2U^c_ng(\log z)}{z} \nonumber \\
    & = \frac{2U^c_n}{z}(g'(\log z) - g(\log z))\nonumber
\end{align}
Now consider $g'(\log z) - g(\log z)$. Its first-order derivative is $\frac{1}{z}(g''(\log z) - g'(\log z)) \leq 0$. Thus, $g'(\log z) - g(\log z)$ is decreasing in $z$, implying $g'(\log z) - g(\log z) \leq g'(0) - g(0) \leq 0$. Putting all together, we have 
\[
g'(\log z)\frac{1}{z^2}\left(-1 + \frac{3U^c_n}{z}\right)  - \frac{2U^c_n g(\log z)}{z^3} \leq 0,
\]
implying that $\cH''_n(z) \leq 0$. So $\cH(z)$ is concave in $z$. As a result $\Psi_n(z)$ is concave in $z$ as desired. 
\end{proof}

\subsection{Proof of Theorem \ref{th:submodular}}

\begin{proof}
The monotonicity is obviously verified as each component  of $\cF(S)$ is monotonically increasing.  For the submodularity, we first see that, if we let $z_n = U^c_n+ \sum_{i\in S}V_{ni}$,  then $\cF(S) = F(\bz)$. To demonstrate submodularity, we adhere to the standard procedure by proving that for any subsets $A$ and $B$ of $[m]$ such that $A\subset B$, and for any $j\in [m]\backslash B$, the following inequality holds:
\begin{equation}\label{eq:concave:eq1}
\cF(A + j) - \cF(A) \geq \cF(B+j) - \cF(B)
\end{equation}
Here, $A+j$ and $B+j$ denote the sets $A\cup {j}$ and $B\cup{j}$, respectively, for ease of notation.
To leverage the concavity of $F(\bz)$ to prove the submodularity, let $\bz^{A},\bz^{B},\bz^{Aj},\bz^{Bj}$ be vectors of size $N$ with elements:
\begin{align*}
    z^A_n &= 1+\sum_{i\in A} V_{ni},~~
    z^B_n = 1+\sum_{i\in B} V_{ni} \\
    z^{Aj}_n &= 1+\sum_{i\in A\cup \{j\}} V_{ni},~~
    z^{Bj}_n = 1+\sum_{i\in B\cup \{j\}} V_{ni}
\end{align*}
We then see that \eqref{eq:concave:eq1} is equivalent to:
\begin{equation}\label{eq:concave-eq2}
    F(\bz^{Aj}) - F(\bz^A) \geq  F(\bz^{Bj}) - F(\bz^B)
\end{equation}
Moreover, since $F(\bz) = \sum_{n\in [N]} \Psi_n(z_n)$, it is sufficient to prove that
\begin{equation}\label{eq:concave-eq*}
      \Psi_n(z^{Aj}_n) - \Psi_n(z^A_n) \geq  \Psi_n(z^{Bj}_n) - \Psi_n(z^B_n)
\end{equation}
Using the  mean value theorem \citep{sahoo1998mean}, there are $\overline{z}^A_n \in [z^A_n, z^{Aj}_n]$ and  $\overline{z}^B_n \in [z^B_n, z^{Bj}_n]$ such that
\begin{align}
    \Psi_n(z^{Aj}_n) - \Psi_n(z^A_n) &= \Psi'_n(\overline{z}^A_n) (z^{Aj}_n - z^{A}_n) = \Psi'_n(\overline{z}^A_n) V_{nj}\label{eq:concave-eq3}\\
  \Psi_n(z^{Bj}_n) - \Psi_n(z^B_n) &= \Psi'_n(\overline{z}^B_n) (z^{Bj}_n - z^{B}_n) = \Psi'_n(\overline{z}^B_n) V_{nj}\label{eq:concave-eq4}
\end{align}
Moreover, since $\Psi_n(z_n)$ is concave in $z_n$, $\Psi'_n(z_n)\leq 0$ for all $z_n>0$, implying that $\Psi'_n(z_n)$ is non-increasing in $z_n$. This follows that:
\begin{align}
    \Psi'_n(\overline{z}^A_n) \geq \Psi'_n(\overline{z}^B_n)
\end{align}
Combine this with \eqref{eq:concave-eq3} and \eqref{eq:concave-eq4} we can validate \eqref{eq:concave-eq*}  and the inequality in \eqref{eq:concave-eq2}, which further confirms the submodularity. We complete the proof. 
\end{proof}

\subsection{Proof of Theorem \ref{th:inner-outer-approx-gap}}

\begin{proof}
     Let $\{(t_1,\Gamma(t_1));\ldots;(t_H,\Gamma(t_H))\}$ be the $H$ breakpoints of $\Gamma^{\OA}$ with a note that $L = t_1<\ldots<t_H = U$.  We construct the following piece-wise linear approximation as
     \[
  \Gamma^{\IA} (t) = \min_{h\in [H-1]} \left\{\Phi(t_h) + \frac{\Phi(t_{h+1}) - \Phi(t_{h})}{t_{h+1} - t_h} (t-t_h) \right\}     
  \]
  To verify the result, we will need to show that (i) $\Gamma^{\IA} (t)$ inner-approximates $\Phi(t)$ and (ii)  the inequality in \eqref{eq:thr-inner-outer} holds. For \textit{(i)}, we leverage the concavity of $\Phi(t)$ to see that, for any $h\in [H-1]$ and $t\in [t_h,t_{h+1}]$, we have
  \begin{align}
      \alpha \Phi(t_{h}) + (1-\alpha)\Phi(t_{h+1}) \leq \Phi(\alpha t_{h} + (1-\alpha)t_{h+1})\label{eq:thr-inner-eq2}
  \end{align}
  where $\alpha = \frac{t_{h+1}- t}{t_{h+1}-t_h}$. Moreover,
  \begin{align*}
      \alpha t_{h} + (1-\alpha)t_{h+1} &= \frac{t_h(t_{h+1}-t)}{t_{h+1}-t_h} + \frac{t_{h+1}(t-t_h)}{t_{h+1}-t_h} = t\\
     \alpha \Phi(t_{h}) + (1-\alpha)\Phi(t_{h+1})  &= \Phi(t_h) + \frac{\Phi(t_{h+1}) - \Phi(t_{h})}{t_{h+1} - t_h} (t-t_h) = \Gamma^{\IA}(t) 
  \end{align*}
  Combine this with \eqref{eq:thr-inner-eq2}, we have $\Phi(t)\geq \Gamma^{\IA}(t)$,  implying that $\Gamma^{\IA}$ inner-approximates $\Phi(t)$ in $[L,U]$.

To prove that the inner-approximation function always yields smaller approximation errors (i.e, inequality \eqref{eq:thr-inner-outer}), we consider an interval $[t_h, t_{h+1}]$ for $h\in [H-1]$.   We will first prove that the following holds true:
\begin{itemize}
    \item[(i)] $\max_{t \in [t_h,t_{h+1}]}|\Phi(t) - \Gamma^{\OA}(t)|  =\max \Big\{\Gamma^{\OA}(t_h) -  \Phi(t_h); \Gamma^{\OA}(t_{t+1}) -  \Phi(t_{h+1}) \Big\}$
    \item[(ii)] $\max_{t \in [t_h,t_{h+1}]}|\Phi(t) - \Gamma^{\IA}(t)|  =  \Phi(t^*) -  \Gamma^{\IA}(t^*) $, where $t^* \in[t_h,t_{h+1}]$ such that $\Phi'(t^*) = \frac{\Phi(t_{h+1}) - \Phi(t_h)}{t_{h+1}-t_h}$ (such $t^*$ always exists due to the mean value theorem)
\end{itemize}

To prove \textit{(i)}, we first see that $\Gamma^{\OA}(t) \geq \Phi(t) \geq \Gamma^{\IA}(t)$, thus, for any $t\in [t_h,t_{h+1}]$, 
\begin{align}
    &\Gamma^{\OA}(t) -  \Phi(t) \leq \Gamma^{\OA}(t)  - \Gamma^{\IA}(t) \\
    &=\Gamma^{\OA}(t_h) + \frac{\Gamma^{\OA}(t_{h+1}) - \Gamma^{\OA}(t_{h})}{t_{h+1} - t_h} (t-t_h)  - \left(\Phi(t_h) + \frac{\Phi(t_{h+1}) - \Phi(t_{h})}{t_{h+1} - t_h} (t-t_h)\right) \\
    &=U_h + \frac{U_{h+1}-U_h}{t_{h+1}-t_h} (t-t_h)\label{eq:thr-inner-eq3}
\end{align}
where 
\begin{align*}
    U_h &= \Gamma^{\OA}(t_h) - \Phi(t_h)\\
    U_{h+1} &= \Gamma^{\OA}(t_{h+1}) - \Phi(t_{h+1})
\end{align*}
Moreover, the function in \eqref{eq:thr-inner-eq3} is linear, implying that:
\[
U_h + \frac{U_{h+1}-U_h}{t_{h+1}-t_h} (t-t_h) \leq \max\left\{ U_{h+1}; U_h\right\} = \max \Big\{\Gamma^{\OA}(t_h) -  \Phi(t_h); \Gamma^{\OA}(t_{t+1}) -  \Phi(t_{h+1}) \Big\}
\]
which confirms \textit{(i)}. 

For \textit{(ii)}, we clearly see that 
\begin{align}
    \Phi(t) - \Gamma^{\IA}(t) &= \Phi(t) - \left(\Phi(t_h) + \Phi'(t^*)(t-t_h)\right)
\end{align}
We now see that the function $\phi(t) = \Phi(t) - \left(\Phi(t_h) + \Phi'(t^*)(t-t_h)\right)$ is concave  in $t$. Taking the first derivative  of $\phi(t)$ we get 
\[
\phi'(t) = \Phi'(t) - \Phi'(t^*)
\]
We then  see that $\phi'(t) = 0$ when $t = t^*$, implying that $\phi(t)$ achieves its maximum  at $t= t^*$. It then follows that:
\[
\phi(t) \leq \phi(t^*) = \Phi(t^*) -\Gamma^{\IA}(t^*)
\]
which confirms \textit{(ii)}.

We now combine $(i)$ and $(ii)$ to see
\begin{align}
    \max_{t \in [t_h,t_{h+1}]}|\Phi(t) - \Gamma^{\IA}(t)|  =  \Phi(t^*) -  \Gamma^{\IA}(t^*) \leq \Gamma^{\OA}(t^*) -  \Gamma^{\IA}(t^*)
\end{align}
We now consider the function $\eta(t) = \Gamma^{\OA}(t) -  \Gamma^{\IA}(t)$. This function is linear in $[t_h,t_{h+1}]$, thus  $\eta(t)\leq \max\{\eta(t_h),\eta({t_{h+1}})\}$, which implies
\begin{align*}
\max_{t \in [t_h,t_{h+1}]}|\Phi(t) - \Gamma^{\IA}(t)| &\leq \max \Big\{\Gamma^{\OA}(t_h) -  \Phi(t_h); \Gamma^{\OA}(t_{t+1}) -  \Phi(t_{h+1}) \Big\}]\\
&= \max_{t \in [t_h,t_{h+1}]}|\Phi(t) - \Gamma^{\OA}(t)|
\end{align*}
 confirming the inequality \eqref{eq:thr-inner-outer}. We complete the proof.
\end{proof}

\subsection{Proof of Theorem \ref{th:inner-milp-bound}}
\begin{proof}
    It can be seen that the approximate MILP in  \eqref{prob:milp-IA} can be rewritten as the following program:
    \begin{align}
     \max_{\bx,\bz} &  \left\{ \widetilde{F}(\bz) =  \sum_{n\in [N]}\Gamma_n(z_n)\right\} \label{eq-thr-proof-approx-prob} \\
    \mbox{subject to} 
    &\quad z_n = \sum_{i\in [m]} x_i V_{ni} + 1,~ \forall n\in [N]\nonumber \\
    &\quad \sum_{i \in [m]} x_i = H \nonumber\\
    &\quad \bx \in \{0,1\}^{m} \nonumber
\end{align}
 Since $\Gamma_n(z_n)$ is an inner-approximation of $\Psi_n(z_n)$, for any $n\in [N]$, we have  $\Gamma_n(z_n)\leq \Psi_n(z_n)$ for any $n\in [N]$. Consequently, $\widetilde{F}(\bz)\leq F(\bz)$ for any $\bz$ in its feasible set. Moreover, the gap between the approximate function $\widetilde{F}(\bz)$ and the true objective function $F(\bz)$ can be bounded as
 \begin{align}
 |\widetilde{F}(\bz) - F(\bz)| &\leq \sum_{n\in [N]} |\Gamma_n(z_n) - \Psi_n(z_n)| \\
 &\leq \sum_{n\in [N]} \max_{z'\in [L_n,U_n]} \left\{|\Gamma_n(z') - \Psi_n(z')|\right\},~~\forall \bz \in\cZ\label{eq:1234}
 \end{align}
where $\cZ$ is the feasible set of $\bz$, defined  as $\cZ = \Big\{\bz \in [L_n;U_n]^n~|~ \exists \bx\in \{0,1\}^m~\text{such that}~ \sum_{i\in [m]}x_i  = C;~ z_n = U^c_n+\sum_{i\in [m]} V_{ni},~\forall n\in [n]\Big\}$ . We now let $(\bx^*,\bz^*)$ be an optimal solution to the true problem \eqref{prob:ME-MCP-main}. We first see that $F(\bz^*) \geq F(\overline{\bz}) \geq \widetilde{F}(\overline{\bz})$. We have the following chain of inequalities:
\begin{align}
    F(\bz^*) -  \widetilde{F}(\overline{\bz})  &\stackrel{(a)}{\leq}  F(\bz^*) - \widetilde{F}(\bz^*) \nonumber\\
    &\stackrel{(b)}{\leq} \sum_{n\in [N]} \max_{z'\in [L_n,U_n]} \left\{|\Gamma_n(z') - \Psi_n(z')|\right\} 
\end{align}
where $(a)$ is because $\bz^*$ is feasible to \eqref{eq-thr-proof-approx-prob} thus $\widetilde{F}(\bz^*) \leq \widetilde{F}(\overline{\bz})$, and $(b)$ is due to the bound in  \eqref{eq:1234}. This confirms the desired inequality \eqref{eq:th-bound-inner-milp} and completes the proof.
\end{proof}
 
\subsection{Proof of Lemma \ref{lm:phi(t)}}

\begin{proof}
    For $(i)$, we take the first-order derivative of $\Theta_n(t)$ to have
    \begin{align*}
        \Theta'(t) &= \frac{\Psi'_n(t)}{(t-a)} - \frac{\Psi_n(t) - \Psi_n(a)}{(t-a)^2}\\
        &=\frac{1}{1-a}\left(\Psi'(t) - \frac{\Psi_n(t) - \Psi_n(a)}{(t-a)}\right)
    \end{align*}
     From the mean value theorem, we know that for any $t>a$, there is $t^a \in (a,t)$ such that $\Psi_n(t^a) = \frac{\Psi_n(t) - \Psi_n(a)}{t -a}$. It  follows that 
     \[
  \Theta'(t) =  \frac{\Psi'_n(t) - \Psi'_n(t^a)}{t-a} \stackrel{(a)}{<} 0 
  \]
where $(a)$ is because  $\Psi_n(t)$ is strictly concave in $t$, thus  $\Psi'_n(t)$ is strictly  decreasing in $t$, implying $\Psi'_n(t) < \Psi'_n(t^a)$. So, we have $\Theta'(t)<0$, so it is strictly  decreasing in $t$.
     
    $(ii)$ is straightforward to verify, as $\Psi_n(z)$ is concave and $\Gamma_n(z)$ is linear in $z$, thus the objective function of \eqref{eq:lambda-t} is concave  in $z$.

For $(iii)$, for a given $t$  such that $t>a$, let $t^a$ be a point in $[a,t]$ such that $\Psi_n(t^a) = \frac{\Psi_n(t) - \Psi_n(a)}{t -a}$. Then,  if we take the first-order derivative of the objective function of \eqref{eq:lambda-t}  and set it to zero, we see  that \eqref{eq:lambda-t} has an optimal solution as $t= t^a$. Consequently, let  $t_1, t_2 \in [a,U]$ such that $t_2>t_1$, and let $t^a_1, t^a_2$ be two points  in $[a,t_1]$ and $[a,t_2]$ such that
\begin{align}
    \Psi'_n(t^a_1) &= \frac{\Psi_n(t_1) - \Psi_n(a)}{t_1 -a} = \Theta_n(t_1);~~   \Psi'_n(t^a_2) = \frac{\Psi_n(t_2) - \Psi_n(a)}{t_2 -a} = \Theta_n(t_2),\nonumber
\end{align}
The above remark implies that 
\begin{align}
    \Lambda_n(t_1|a) &= \Psi_n(t^a_1) -  \Psi_n(a)  - \frac{\Psi_n(t_1) -\Psi_n(a) }{t_1-a} (t^a_1 - a) = \Psi_n(t^a_1) - \Theta_n(t_1)(t^a_1-a) - \Psi_a(a) \nonumber\\
    &= \Psi_n(t^a_1) - \Psi'_n(t^a_1)(t^a_1-a) - \Psi_a(a)\label{eq:x1} \\
  \Lambda_n(t_2|a) &= \Psi_n(t^a_2) -  \Psi_n(a)  - \frac{\Psi_n(t_2) -\Psi_n(a) }{t_2-a} (t^a_2 - a) = \Psi_n(t^a_2) - \Theta_n(t_2)(t^a_2-a) - \Psi_a(a)\nonumber \\
  &= \Psi_n(t^a_2) - \Psi'_n(t^a_2)(t^a_2-a) - \Psi_a(a)\label{eq:x2} 
\end{align}
Moreover,  we  observe that,  since $\Theta_n(t)$ is (strictly) decreasing in $t$, $\Psi'_n(t^a_1) > \Psi'_n(t^a_2)$. Combine this with the fact that $\Psi'_n(t)$ is  (strictly) decreasing in $t$, we have $t_1^a <t^a_2$. To prove that $\Lambda_n(t_2|a)>\Lambda_n(t_1|a)$, let us consider the following function:
\[
U(t) = \Psi_n(t) - \Psi'_n(t)(t-a)
\]
Taking the first-order derivative of $U(t)$ w.r.t. $t$ we get
\[
U'(t) = \Psi'_n(t) - \Psi'_n(t)  - \Psi^{''}_n(t)(t-a) \stackrel{(b)}{>0}, ~ \forall t>a
\]
where $(b)$ is because $\Psi^{''}_n(t)<0$ (it is strictly concave in $t$). So, $U(t)$ is (strictly) increasing in $t$, implying:
\[
U(t^a_1) <U(t^a_2)
\]

Combine this with \eqref{eq:x1} and \eqref{eq:x2} we get  $\Lambda_n(t_1|a) < \Lambda_n(t_2|a)$ as desired.
\end{proof}

\subsection{Proof of Theorem \ref{th:breakpoints}}

\begin{proof}
 To prove $(i)$,  by contradiction  let us assume that $\max_{k\in [K]} \Lambda_n(c'_{k+1}|c'_{k}) \leq \epsilon$ (denoted as Assumption (\textbf{A}) for  later reference). Under this assumption, let us choose $k$ as the first index in $\{1,\ldots,K+1\}$ such that $c'_k \neq c^n_k$ (i.e., $c'_h = c^n_h$ for all $1\leq h<k$).  Such an index always exists as $K<K_n$.
 We consider two cases:
 \begin{itemize}
     \item If $c'_k > c^n_k$, then from  the monotonicity of the function $\Lambda_n(t|c^n_{k-1})$, we should have 
     $$\Lambda_n(c'_k|c^n_{k-1}) >\Lambda_n(c^n_{k+1}|c^n_k)  = \epsilon$$
          which violates Assumption (\textbf{A}).
          \item If $c'_k < c^n_k$, then if $c^n_{k+1} \neq U_n$ we should have $\Gamma(c^n_{k+1}|c'_{k}) >\Gamma(c^n_{k+1}|c^n_{k}) = \epsilon$. Consequently, to ensure that (\textbf{A}) holds, we need $c'_{k+1} <c^n_{k+1}$.
 \end{itemize}
So, we must have $c'_n \leq c^n_k$, for all $k\in [K+1]$, implying that $K\geq K_n$, contradicting to the initial  assumption that $K<K_n$. So, the contradiction assumption (\textbf{A}) must be false, as desired.

For bounding $K_n$, for any $k\in [K_n]$, we take the middle point of $[c^n_k,c^n_{k+1}]$ to bound  $\Lambda_n(c^n_{k+1}|c^n_{k})$ from below as 
\begin{align}
    \Lambda_n(c^n_{k+1}|c^n_{k}) &= \max_{z\in [c^n_{k}, c^n_{k+1}]} \left\{ \Psi_n(z) - \Gamma_n(z)\right\} \nonumber \\
    &\geq \Psi_n\left(\frac{c^n_{k}+ c^n_{k+1}}{2}\right) - \Gamma_n\left(\frac{c^n_{k}+ c^n_{k+1}}{2}\right) \nonumber\\
    &= \Psi_n\left(\frac{c^n_{k}+ c^n_{k+1}}{2}\right) - \frac{1}{2}\left(\Psi_n(c^n_{k+1}) + \Psi_n(c^n_{k+1})\right) 
\end{align}
 According to the \textit{Second-order Mean Value Theorem} \citep{stewart2015calculus}, there is $c \in [c^n_{k}, c^n_{k+1}]$ such that 
 \[
 \Psi_n\left(\frac{c^n_{k}+ c^n_{k+1}}{2}\right) - \frac{1}{2}\left(\Psi_n(c^n_{k+1}) + \Psi_n(c^n_{k+1})\right)  = \frac{1}{4}(c^n_{k_1}- c^n_{k})^2 |\Psi''_n(c)|
 \]
 Combine this with the fact that $\Lambda_n(c^n_{k+1}|c^n_{k}) \leq \epsilon$, we should have 
 \begin{align}
     \frac{1}{4}(c^n_{k+1}- c^n_{k})^2 \Psi''_n(c) \leq \epsilon,
 \end{align}
 implying that  
 \[
 c^n_{k+1} - c^n_k \leq \sqrt{\frac{4\epsilon}{|\Psi''_n(c)|}} \leq 2\sqrt{\frac{{\epsilon}}{L^{\Psi}_n}}.
 \]
 Using this, we write
 \[
 U_n - L_n = \sum_{k\in [K_n]} (c^n_{k+1} - c^n_k) \leq  2(K_n)\sqrt{\frac{{\epsilon}}{L^{\Psi}_n}},
 \]
 or equivalently,
 \[
 K_n \geq\frac{(U_n-L_n)\sqrt{L^\Psi_n}}{2\sqrt{\epsilon}},
 \]
 which confirms the lower bound. 

 For the upper-bounding  $K_n$, let us consider $k\leq K_n$. From the way $c^n_k$ are selected, we have:
 \begin{align}
\epsilon = \Lambda_n(c^n_{k+1}|c^n_{k}) &= \max_{z\in [c^n_{k}, c^n_{k+1}]} \left\{ \Psi_n(z) - \Gamma_n(z)\right\}\nonumber \\
&\stackrel{(a)}{\leq} \Psi_n(c^n_k) +  \Psi'_n(c^n_{k}) (c^n_{k+1} - c^n_{k})  - \Psi_n(c^n_{k+1}) \label{eq:bound-eq1} 
 \end{align}
 where $(a)$ is because for  any $z\in [c^n_k,c^n_{k+1}]$ we have: $\Psi_n(z)\leq \Psi_n(c^n_k) +  \Psi'_n(c^n_{k}) (c^n_{k+1} - c^n_{k})$ (as $\Psi_n(\cdot)$ is concave), thus $\Psi_n(z) - \Gamma_n(z) \leq \Psi_n(c^n_k) +  \Psi'_n(c^n_{k}) (c^n_{k+1} - c^n_{k}) - \Gamma_n(z) \leq \Psi_n(c^n_k) +  \Psi'_n(c^n_{k}) (c^n_{k+1} - c^n_{k})  - \Psi_n(c^n_{k+1})$.
 Moreover, it follows from\textit{ Taylor's theorem} that, there is  $c \in [c^n_{k}, c^n_{k+1}]$ such that
 \[
 \Psi_n(c^n_{k+1}) = \Psi_n(c^n_k) +  \Psi'_n(c^n_{k}) (c^n_{k+1} - c^n_{k}) + \frac{(c^n_{k+1} - c^n_k)^2}{2} \Psi''_n(c).
 \]
 Combine this with \eqref{eq:bound-eq1} we get 
 \[
 \epsilon \leq \frac{(c^n_{k+1} - c^n_k)^2}{2} \Psi''_n(c) \leq \frac{(c^n_{k+1} - c^n_k)^2}{2} U^{\Psi}_n,
 \]
implying  that:  \[
 c^n_{k+1} - c^n_k   \geq \sqrt{\frac{{2\epsilon}}{U^{\Psi}_n}}.
 \]
Now, similar to the establishment of the lower bound, we write:
 \[
 U_n - L_n \geq  \sum_{k\in [K_n]} (c^n_{k+1} - c^n_k) \geq  K_n\sqrt{\frac{{2\epsilon}}{U^{\Psi}_n}}
 \]
 leading to
 \[
 K_n \leq \frac{(U_n-L_n)\sqrt{U^\Psi_n}}{\sqrt{2\epsilon}},
 \]
as desired. 
\end{proof}

\subsection{Proof of Theorem \ref{th:non-concave-bound}}

\begin{proof}
We first see that \eqref{prob:milp-2} is equivalent to the following mixed-integer nonlinear program
\begin{align}
     \max_{\bx,\bz} &\left\{\widetilde{F}(\bz) = \sum_{n\in [N]}  \Gamma_n(z_n)\right\} \label{proof:milp-2} \\
    \mbox{subject to} 
    &\quad z_n = \sum_{i\in [m]} x_i V_{ni} + 1,~\forall n\in [N] \nonumber\\
    &\quad \sum_{i \in [m]} x_i \leq C \nonumber\\
    &\quad x_i \in \{0,1\}, ~\forall n \in [N], k \in [K_n] \nonumber
\end{align}

where $\Gamma_n(z_n)$ , $\forall n \in [N]$, are defined in \eqref{eq:Gamma-nonconcave}. The equivalence can be seen easily: if $(\bx,\by,\bz,\br)$ is a feasible solution to \eqref{prob:milp-2}, then $(\bx,\bz)$ is also feasible and yields the same objective value for \eqref{proof:milp-2}. Conversely, if $(\bx,\bz)$ is feasible to \eqref{proof:milp-2}, then for each $n \in [N]$, let $k^n$ be the maximum index in $[K_n]$ such that $c^n_{k^n} \leq z_n$. We then choose $\by$ and $\br$ such that
\[y_{nk} = 
\begin{cases}
   1 & \text{if } k \leq k^n \\
   0 & \text{otherwise}
\end{cases}
\text{ and }
r_{nk} = \begin{cases}
    y_{nk} & \text{if } k \leq k^n \\
    0 & \text{if } k \geq k^n + 2 \\
    z_n - c^n_{k^n} & \text{if } k = k^n + 1
\end{cases}
\]
Then we see that $(\bx,\by,\bz,\br)$ is feasible to \eqref{prob:milp-2}. This solution also gives the same objective value as the one given by $(\bx,\bz)$ in \eqref{proof:milp-2}. All these imply the equivalence.

So, if \((\overline{\bx}, \overline{\by}, \overline{\bz}, \overline{\br})\) is an optimal solution to \eqref{prob:milp-2}, then \((\overline{\bx}, \overline{\bz})\) is also optimal for \eqref{proof:milp-2}. Moreover, \((\overline{\bx}, \overline{\bz})\) is feasible to the original problem \eqref{prob:ME-MCP-main}. These lead to the following inequalities:
\begin{align}
      F(\overline{\bz}) &\stackrel{(a)}{\leq} F(\bx^*) \\
      &\stackrel{(b)}{\leq}  \widetilde{F}(\bz^*) + N\epsilon \\
      &\stackrel{(c)}{\leq} \widetilde{F}(\overline{\bz}) + N\epsilon
\end{align}
where $(a)$ is because $(\bx,\bz)$ is a feasible solution to the ME-MCP problem in \eqref{prob:ME-MCP-main}, $(b)$ is because of the assumption $|\Psi_n(z_n) - \Gamma_n(z_n)|\leq \epsilon$, which directly implies $|F(\bz) - \widetilde{F}(\bz)| \leq N\epsilon$, and $(c)$ is because $(\bx^*,\bz^*)$ is also feasible to \eqref{proof:milp-2}. These inequalities directly imply:
 \[
 |F(\overline{\bz}) - F(\bz^*)| \leq N\epsilon.
 \]
 as desired.
\end{proof}

\subsection{Proof of Proposition \ref{prop:finding-breakpoint-nonconcave}}

\begin{proof}
For any $n \in [N]$ and interval $[a, b]$ where $\Psi_n(z)$ is either concave or convex, according to [citation], the number of breakpoints generated within this interval can be bounded by:
\[
\frac{(b - a) \sqrt{U^\Psi_n}}{\sqrt{2\epsilon}}
\]
Let $\{[a_1, b_1], [a_2, b_2], \ldots, [a_T, b_T]\}$ be the $T$ sub-intervals generated by the [Finding Breakpoints] procedure. The number of breakpoints within $[L_n, U_n]$ can be bounded as:
\[
K_n \leq \sum_{t \in [T]} \left(\frac{(b_t - a_t) \sqrt{U^\Psi_n}}{\sqrt{2\epsilon}} \right) \leq \frac{(U_n - L_n) \sqrt{U^\Psi_n}}{\sqrt{2\epsilon}}.
\]

as desired.

\end{proof}

\subsection{Proof of Theorem \ref{th:non-concave-reducedMILP}}
\begin{proof}
   We first need the following lemma to prove the claim:   \begin{lemma}\label{lm:non-concave}
    Given $n\in [N]$ and a sub-interval $[a,b]\subset [L_n,U_n]$, assume that $\Psi_n(z)$ is concave in$[a,b]$. Let  $\{c^n_u,\ldots,c^n_v\}$ be the breakpoints generated within $[a,b]$, i.e., $a=c^n_u <c^n_{u+1}<\ldots<c^n_v = b$, we have  $\gamma^n_u\geq  \gamma^n_{u+1} \geq ...\geq \gamma^n_{v-1}$. 
    \end{lemma}
The lemma  can be easily verified  by recalling  that, for any $u\leq j\leq v-2$ we have
\begin{align}
     \gamma^n_j &= \frac{\Psi_n(c^n_{j+1}) - \Psi_n(c^n_{j}) }{c^n_{j+1}-c^n_{j}}\nonumber\\
      \gamma^n_{j+1} &= \frac{\Psi_n(c^n_{j+2}) - \Psi_n(c^n_{j+1}) }{c^n_{j+2}-c^n_{j+1}}\nonumber
\end{align}
So, from the Mean Value Theorem, there are $d^n_{j} \in [c^n_{j}, c^n_{j+1}]$ and $d^n_{j+1} \in [c^n_{j+1}, c^n_{j+2}]$ such that 
\[
   \gamma^n_j = \Psi'_n(d^n_j);~ \gamma^n_{j+1} = \Psi'_n(d^n_{j+1})
\]
Moreover, since $\Psi_n(z)$ is concave in $[a,b]$, its  first-order derivative $\Psi'_n(z)$ is decreasing in $[a,b]$, implying $\gamma^n_j\geq \gamma^n_{j+1}$. We verified the lemma then.
    
We now return to the main result. We will show that there is an optimal solution to \eqref{prob:milp-2} which is also feasible to \eqref{prob:milp-relax}, directly implying the equivalence. Let $(\bx^*, \by^*, \bz^*, \br^*)$ be an optimal solution to \eqref{prob:milp-relax}. As discussed earlier, for each $n \in [N]$, the breakpoints are constructed by dividing the interval $[L_n, U_n]$ into sub-intervals within which $\Psi_n(z)$ is either concave or convex. 

Consider an interval $[a, b]$ where $\Psi_n(z)$ is concave and assume that within this interval we can generate breakpoints $a = c^n_u, c^n_{u+1}, \ldots, c^n_v = b$ (for $1 \leq u < v \leq K_n+1$). We consider the following cases:
\begin{itemize}
    \item \textbf{Case 1:} If there is $u' < u$ such that $y^*_{nu'} = 0$, then from Constraints \ref{ctr:ynk} we see that $y^*_{nk} = 0$ for $k \in \{u, \ldots, v-1\}$, which are binary values.
    \item \textbf{Case 2:} If there is $v' \geq v$ such that $y^*_{nv'} = 1$, then Constraints \ref{ctr:ynk} imply $y^*_{nk} = 1$ for $k \in \{u, \ldots, v-1\}$, which are also binary values.
\end{itemize}

Now consider \textbf{Case 3} where $y^*_{nu'} = 1$ for any $u' < u$, and $y^*_{nv'} = 0$ for any $v' \geq v$. For some extreme cases where $u' < 1$, set $y_{nu'} = 1$; and if $v' > K_n+1$, set $y^*_{nv'} = 0$. We will show that from the optimal solution, we can construct an optimal solution $(\bx^{**}, \by^{**}, \bz^{**}, \br^{**})$ where $y^{**}_{nk}$ take binary values for all $k,n$.

To this end, within the set $\{u, u+1, \ldots, v-1\}$, if we can find two indices $u_1, v_1$ such that $u_1 < v_1$ and $r^*_{n,u_1} < 1$ and $r^*_{n,v_1} > 0$, due to the properties stated in Lemma \ref{lm:non-concave}, we can always decrease $r_{n,v_1}$ and increase $r^*_{n,u_1}$ to get a better (or  at least similar) objective value while keeping Constraints \eqref{ctr:r-x} satisfied. Specifically, we can subtract $r^*_{n,v_1}$ by $\epsilon/(c^n_{v_1+1} - c^n_{v_1})$ and increase $r^*_{n,u_1}$ by $\epsilon/(c^n_{u_1+1} - c^n_{u_1})$ ($\epsilon > 0$ is chosen such that the new values of $r^*_{n,u_1}$ and $r^*_{n,v_1}$ are still within $[0,1]$). By doing so, we can obtain a better (or at least as good as the current optimal values):

\[
\begin{aligned}
    &\gamma^n_{u_1} (c^n_{u_1+1} - c^n_{u_1})\left(r^*_{n,u_1} + \frac{\epsilon}{c^n_{u_1+1} - c^n_{u_1}}\right) + \gamma^n_{v_1} (c^n_{v_1+1} - c^n_{v_1})\left(r^*_{n,v_1} - \frac{\epsilon}{c^n_{v_1+1} - c^n_{v_1}}\right) \\
    &= \gamma^n_{u_1} (c^n_{u_1+1} - c^n_{u_1})r^*_{n,u_1} + \gamma^n_{v_1} (c^n_{v_1+1} - c^n_{v_1})r^*_{n,v_1} + \epsilon(\gamma^n_{u_1} - \gamma^n_{v_1}) \\
    &\stackrel{(a)}{\geq} \gamma^n_{u_1} (c^n_{u_1+1} - c^n_{u_1})r^*_{n,u_1} + \gamma^n_{v_1} (c^n_{v_1+1} - c^n_{v_1})r^*_{n,v_1}
\end{aligned}
\]

where $(a)$ is because $\gamma^n_{u_1} \geq \gamma^n_{v_1}$ (Lemma \ref{lm:non-concave}). Moreover, we can see that Constraints \eqref{ctr:r-x} are still satisfied with the new values as

\[
\begin{aligned}
    &(c^n_{u_1+1} - c^n_{u_1})\left(r^*_{n,u_1} + \frac{\epsilon}{c^n_{u_1+1} - c^n_{u_1}}\right) + (c^n_{v_1+1} - c^n_{v_1})\left(r^*_{n,v_1} - \frac{\epsilon}{c^n_{v_1+1} - c^n_{v_1}}\right) \\
    &= (c^n_{u_1+1} - c^n_{u_1})r^*_{n,u_1} + (c^n_{v_1+1} - c^n_{v_1})r^*_{n,v_1}
\end{aligned}
\]

So, we can always adjust $r^*_{nk}$ for $k \in \{u, \ldots, v-1\}$, in such a way that for any indices $u_1, v_1$ such that $u \leq u_1 < v_1 \leq v-1$, we have either $r^*_{n,u_1} = 1$ or $r^*_{n,v_1} = 0$. These new values give at least as good objective values as the old ones, while ensuring that Constraints \eqref{ctr:r-x} are still satisfied. For these adjusted values, there is an index $\tau \in \{u, \ldots, v-1\}$ such that $r^*_{nt} = 1$ for all $t$ such that $u \leq t < \tau$ and $r^*_{nt} = 0$ for all $\tau < t \leq v-1$. For this new value, we can also adjust the variable $y^*_{nk}$, for $k \in \{u, \ldots, v-1\}$, such that $y^*_{nt} = 1$ for all $u \leq t < \tau$, and $y^*_{nt} = 0$ for all $\tau \leq t \leq v$. We can easily verify that the adjusted solutions still satisfy all the constraints in \eqref{prob:milp-relax}.

We now apply this adjustment for all concave intervals $[a, b]$ and all $n \in [N]$ to obtain a new adjusted solution $(\overline{\bx}, \overline{\by}, \overline{\bz}, \overline{\br})$ that is feasible to \eqref{prob:milp-relax} while offering at least as good an objective value as the one given by $(\bx^*, \by^*, \bz^*, \br^*)$. Moreover, since $(\bx^*, \by^*, \bz^*, \br^*)$ is optimal for \eqref{prob:milp-relax}, the adjusted solution $(\overline{\bx}, \overline{\by}, \overline{\bz}, \overline{\br})$ is also optimal for this problem. Additionally, $\overline{\by}$ is a binary vector, thus $(\overline{\bx}, \overline{\by}, \overline{\bz}, \overline{\br})$ is also feasible for the original problem \eqref{prob:milp-2} (the problem before variables $\by$ are partially relaxed). All this implies the equivalence between \eqref{prob:milp-2} and the relaxed version \eqref{prob:milp-relax}, as desired.

\end{proof}

\section{ ``Inner-approximation'' for Convex Functions}\label{appdx:inner-convex}

In this section we describe how to apply the techniques  in Section \ref{sec:outer-inner} to construct a piece-wise linear approximation of  $\Psi_n(z)$, in the case that $\Psi_n(z)$ is convex in $z$.  That is, let us assume that function $\Psi_n(z)$ is convex in $[L_n,U_n]$ and our aim is to approximate it by a convex piece-wise linear function of the form
\[
\widetilde{\Gamma}_n(z) =\max_{k\in [K_n-1]} \left\{\Psi_n(c^n_k) + \frac{\Psi_n(c^n_{k+1}) - \Psi_n(c^n_{k})}{c^n_{k+1} - c^n_k} (z-c^n_k) \right\},~ \forall n\in[N].  
\]
where $c^n_k,~ k\in [K_n]$ are $K_n$ breakpoints. Here, it can be seen that, $\widetilde{\Gamma}_n(z)$ outer-approximates $\Psi_n(z)$, instead of inner-approximating this function in the case that $\Psi_n(z)$ is convex. %We provide an illustration in Figure [] below. \mtien{[Insert a figure here ...]}

Now,  we describe  our method to generate the breakpoints $c^n_k,\ldots,c^n_{K_n}$ such that $\max_{z\in [L_n,U_n]} \{\widetilde{F}_n(z) - \Psi_n(z)\} \leq \epsilon$, while the number of breakpoints $K_n$ is minimized.   Similar to the concave situation, let us define the following functions
\begin{align}
  \widetilde{\Lambda}_n(t|a) &= \max_{z\in [a,t]}\left\{ \widetilde{\Gamma}_n(z)- {\Psi}_n(z)\right\}    \label{eq:lambda-t-convex} \\
  \wTheta_n(t)&= \frac{\Psi_n(t) -\Psi_n(a) }{t-a} 
\end{align}
We have the following results

\begin{lemma}
\label{lm:phi(t)-convex}
The following results hold
\begin{itemize}
    \item [(i)] $\Theta_n(t)$ is (strictly) increasing in $t$ 
   \item [(ii)] $\wLambda_n(t|a)$ can be computed by convex optimization 
   \item [(iii)] $\wLambda_n(t|a)$ is strictly monotonic increasing in $t$, for any $t\geq a$.
\end{itemize}

\end{lemma}
\begin{proof}
  The proof can be done similarly as the  proof of Lemma [], we first take the first-order derivative of $\wTheta_n(t)$ to see
   \begin{align*}
        \wTheta'(t) &= \frac{\Psi'_n(t)}{(t-a)} - \frac{\Psi_n(t) - \Psi_n(a)}{(t-a)^2}\\
        &=\frac{1}{1-a}\left(\Psi'(t) - \frac{\Psi_n(t) - \Psi_n(a)}{(t-a)}\right)
    \end{align*}
     From the mean value theorem, we know that for any $t>a$, there is $t^a \in (a,t)$ such that $\Psi_n(t^a) = \frac{\Psi_n(t) - \Psi_n(a)}{t -a}$. It  follows that 
     \[
  \wTheta'(t) =  \frac{\Psi'_n(t) - \Psi'_n(t^a)}{t-a} \stackrel{(a)}{<} 0 
  \]
where $(a)$ is because  $\Psi_n(t)$ is strictly convex in $t$, thus  $\Psi'_n(t)$ is strictly  increasing in $t$, implying $\Psi'_n(t) > \Psi'_n(t^a)$. So, we have $\wTheta'(t)>0$, so it is strictly  increasing in $t$.
     
    $(ii)$ is straightforward to verify, as $\Psi_n(z)$ is convex and $\Gamma_n(z)$ is linear in $z$, thus the objective function of \eqref{eq:lambda-t-convex} is concave  in $z$.

For $(iii)$, for a given $t$  such that $t>a$, let $t^a$ be a point in $[a,t]$ such that $\Psi_n(t^a) = \frac{\Psi_n(t) - \Psi_n(a)}{t -a}$. Then,  if we take the first-order derivative of the objective function of \eqref{eq:lambda-t-convex}  and set it to zero, we see  that \eqref{eq:lambda-t-convex} has an optimal solution as $t= t^a$. Consequently, let  $t_1, t_2 \in [a,U]$ such that $t_2>t_1$, and let $t^a_1, t^a_2$ be two points  in $[a,t_1]$ and $[a,t_2]$ such that
\begin{align}
    \Psi'_n(t^a_1) &= \frac{\Psi_n(t_1) - \Psi_n(a)}{t_1 -a} = \wTheta_n(t_1);~~   \Psi'_n(t^a_2) = \frac{\Psi_n(t_2) - \Psi_n(a)}{t_2 -a} = \wTheta_n(t_2),\nonumber
\end{align}
The above remark implies that 
\begin{align}
    \wLambda_n(t_1|a) &= \Psi_n(a)  + \frac{\Psi_n(t_1) -\Psi_n(a) }{t_1-a} (t^a_1 - a) - \Psi_n(t^a_1)  = - \Psi_n(t^a_1) + \wTheta_n(t_1)(t^a_1-a) + \Psi_a(a) \nonumber\\
    &= -\Psi_n(t^a_1) + \Psi'_n(t^a_1)(t^a_1-a) + \Psi_a(a)\label{eq:x1-v2} \\
  \wLambda_n(t_2|a) &=  -\Psi_n(t^a_2) +  \Psi_n(a)  + \frac{\Psi_n(t_2) -\Psi_n(a) }{t_2-a} (t^a_2 - a) = -\Psi_n(t^a_2) + \wTheta_n(t_2)(t^a_2-a) + \Psi_a(a)\nonumber \\
  &= -\Psi_n(t^a_2) + \Psi'_n(t^a_2)(t^a_2-a) +\Psi_a(a)\label{eq:x2-v2} 
\end{align}
Moreover,  since $\wTheta_n(t)$ is (strictly) increasing in $t$, $\Psi'_n(t^a_1) 
<\Psi'_n(t^a_2)$. Combine this with the fact that $\Psi'_n(t)$ is  (strictly) increasing in $t$, we have $t_1^a <t^a_2$. To prove that $\wLambda_n(t_2|a)>\wLambda_n(t_1|a)$, let us consider the following function:
\[
U(t) = \Psi'_n(t)(t-a) -  \Psi_n(t) 
\]
Taking the first-order derivative of $U(t)$ w.r.t. $t$ we get
\[
U'(t) = - \Psi'_n(t) + \Psi'_n(t) + \Psi^{''}_n(t)(t-a) \stackrel{(b)}{>0}, ~ \forall t>a
\]
where $(b)$ is because $\Psi^{''}_n(t)>0$ (it is strictly convex in $t$). So, $U(t)$ is (strictly) increasing in $t$, implying:
\[
U(t^a_1) <U(t^a_2)
\]

Combine this with \eqref{eq:x1-v2} and \eqref{eq:x2-v2} we get  $\wLambda_n(t_1|a) < \wLambda_n(t_2|a)$ as desired.
  
\end{proof}

Thanks to the assertions in Lemma \ref{lm:phi(t)-convex}, we can derive the breakpoints $c^n_1,\ldots,c^n_{K_n}$ following a procedure akin to that outlined in Section \ref{subsec:concave-opt-points}. Initially, we set the first point as $c^n_1 = L_n$. At each breakpoint $c^n_k$, the subsequent breakpoint $c^n_{k+1}$ can be efficiently determined by solving the optimization problem:
\[
c^n_{k+1} = \text{argmax}_{z\in [c^n_k,U_n]} \{\wLambda(z|c^n_k) \leq \epsilon\}
\]
This can be achieved through binary search, with each step involving solving a simple univariate convex optimization problem. Thanks to claim $(ii)$ of Lemma \ref{lm:phi(t)-convex}, we ascertain that such a next breakpoint will be uniquely determined, and, except for the last breakpoint, we should have $\wLambda(c^n_{k+1}|c^n_k) = \epsilon$. Consequently, this implies the optimality of the number of breakpoints required to achieve the desired approximation error, similar to the assertions in Theorem \ref{th:breakpoints}. Specifically, utilizing similar arguments, we can establish that any piece-wise linear approximation with a smaller number of breakpoints will inevitably result in a larger approximation error.

\end{document}

$$\Phi_n(\bx) = q_ng(\phi_n(\bx))\left(\frac{\sum_{i\in[m]}x_iV_{in}}{\sum_{i\in[m]}x_iV_{in} + U^c_n}\right) + \alpha_n \Psi_n(\bx)$$

\begin{align*}
    \frac{\partial \Psi_n(\bx)}{\partial x_j} &= q_n \frac{\partial g(\phi_n(\bx))}{\partial \Psi_n(\bx)} \frac{\partial \Psi_n(\bx)}{\partial x_j}\frac{\sum_{i\in[m]}x_iV_{in}}{\sum_{i\in[m]}x_iV_{in} + U^c_n} \\ &\quad\quad+ q_ng(\phi_n(\bx))\frac{\partial}{\partial x_j}\frac{\sum_{i\in[m]}x_iV_{in}}{\sum_{i\in[m]}x_iV_{in} + U^c_n} + \alpha_n \frac{\partial \Psi_n(\bx)}{\partial x_j}\\
    &= q_n\alpha\beta\exp(-\alpha\phi_n(\bx))\frac{V_{nj}}{\sum_{i\in[m]}x_iV_{ni} + U^c_n}\frac{\sum_{i\in[m]}x_iV_{in}}{\sum_{i\in[m]}x_iV_{in} + U^c_n} \\
    &\quad\quad + q_ng(\phi_n(\bx))\left(\frac{V_{nj}}{\sum_{i\in[m]}x_iV_{ni} + U^c_n} - \frac{V_{nj}\sum_{i\in[m]}x_iV_{ni}}{(\sum_{i\in[m]}x_iV_{ni} + U^c_n)^2}\right) + \alpha_n\frac{V_{nj}}{\sum_{i\in[m]}x_iV_{ni} + U^c_n}\\
    &=q_n(\alpha\beta\exp(-\alpha\phi_n(\bx)) - g(\phi_n(\bx)))\frac{V_{nj}\sum_{i\in[m]}x_iV_{ni}}{(\sum_{i\in[m]}x_iV_{ni} + U^c_n)^2}\\
    &\quad\quad + (q_ng(\phi_n(\bx)) + \alpha_n)\frac{V_{nj}}{\sum_{i\in[m]}x_iV_{ni} + U^c_n}.
\end{align*}
\end{document}